\documentclass[11pt,leqno]{amsart}
\usepackage[utf8]{inputenc}
\usepackage{amssymb,graphicx,xcolor} 
\usepackage{mathrsfs}
\usepackage[hidelinks]{hyperref} 
\usepackage{upgreek}
\usepackage{datetime}
\usepackage{appendix}

\begin{document}
\theoremstyle{plain}
\newtheorem{thm}{Theorem}[section]
\newtheorem{prop}[thm]{Proposition}
\newtheorem{lem}[thm]{Lemma}
\newtheorem{cor}[thm]{Corollary}
\newtheorem{deft}[thm]{Definition}
\newtheorem{hyp}{Assumption}
\newtheorem*{KSU}{Theorem (Kenig, Sj\"ostrand and Uhlmann)}

\theoremstyle{definition}
\newtheorem{rem}[thm]{Remark}
\numberwithin{equation}{section}
\newcommand{\eps}{\varepsilon}
\renewcommand{\d}{\partial}
\newcommand{\dd}{\mathrm{d}}
\newcommand{\e}{\mathrm{e}}
\newcommand{\re}{\mathop{\rm Re} }
\newcommand{\im}{\mathop{\rm Im}}
\newcommand{\ch}{\mathop{\rm ch}}
\newcommand{\R}{\mathbf{R}}
\newcommand{\C}{\mathbf{C}}
\renewcommand{\H}{\mathbf{H}} 
\newcommand{\N}{\mathbf{N}} 
\newcommand{\D}{\mathcal{C}^{\infty}_0} 
\newcommand{\supp}{\mathop{\rm supp}}

\newcommand{\tre}{\textcolor{red}}
\newcommand{\tbl}{\textcolor{blue}}
\hyphenation{pa-ra-met-ri-zed}

\title[]{Stability estimates for the Magnetic Schr\"odinger operator with partial measurements}
\author[]{Leyter Potenciano-Machado, Alberto Ruiz \and Leo Tzou}
\address{University of Jyv\"askyl\"a, Department of Mathematics and Statistics, PO Box 35, 40014, Finland}
\email{leyter.m.potenciano@jyu.fi \qquad leymath@gmail.com}
\address{Departamento de Matem\'aticas, Universidad Aut\'onoma de Madrid, Campus de Cantoblanco, 28049 Madrid, Spain}
\email{alberto.ruiz@uam.es}
\address{School of Mathematics and Statistics, University of Sydney, NSW 2006, Australia}
\email{leo.tzou@gmail.com}

\begin{abstract}
In this article, we study stability estimates when recovering magnetic fields and electric potentials in a simply connected open subset in $\mathbb{R}^n$ with $n\geq 3$, from measurements on open subsets of its boundary. This inverse problem is associated with a magnetic Schr\"odinger operator. Our estimates are quantitative versions of the uniqueness results obtained by D.~Dos Santos Ferreira,  C.~E.~Kenig, J.~Sj\"ostrand and G.~Uhlmann in \cite{DSFKSjU}. The moduli of continuity are of logarithmic type.
\end{abstract}
\maketitle
\setcounter{tocdepth}{1} 
\tableofcontents

\section{Introduction}

Let $\Omega\subset \mathbb{R}^n$ ($n\geq 3$) be a bounded open and connected set with smooth boundary $\partial\Omega$. We consider the following magnetic Schr\"odinger operator:
\[
\mathcal{L}_{A,q}: = D^2 + A\cdot D + D\cdot A + A^2 +q,
\]
where  $A^2=A\cdot A$, $D=-i \nabla$, $D^2=D\cdot D$, and $A=\left(  A_j \right)_{j=1}^{n}\in W^{1,\infty}\left(\Omega; \mathbb{C}^n \right)$ and $q\in L^{\infty}\left( \Omega; \mathbb{C} \right) $ denote the magnetic and electric potentials, respectively. Throughout this manuscript, we consider $A$ as a $1$-form and then $dA$ as a $2$-form, more precisely: 
\[
A=\sum_{j=1}^{n}A_jdx_j \quad \text{and}\quad dA= \sum_{1\leq j<k\leq n}\left( \partial_{x_j}A_k-\partial_{x_k}A_j  \right) d_{x_j}\wedge d_{x_k}.
\]

The inverse boundary value problem (IBVP) under consideration in this article is to recover information (in $\Omega$) about the magnetic field $dA$ and the electric potential $q$ from voltage and current measurements on open subsets of $\partial \Omega$. To describe our results, we denote by $F$ and $B$ two arbitrary and nonempty subsets of $\partial\Omega$. The local boundary measurements are captured in the partial Dirichlet-Neumann (DN) map:
\[
\begin{matrix}
 \Lambda_{A,q}^{B\rightarrow F}:&H^{1/2}(B) &\rightarrow&H^{-1/2}(\partial\Omega) \\ 
 &f& \mapsto&  (\nu\cdot (\nabla+ i\, A)u_f)|_{F}, 
\end{matrix}
\]
where $\nu(x)$ is the outer unit normal of $x\in \partial\Omega$, the set $H^{1/2}(B)$ denotes the space consisting of all $f\in H^{1/2}(\partial\Omega)$ such that $\supp(f)\subset B$ and $u_f\in H^1(\Omega)$ solves the equation $\mathcal{L}_{A,q}\,u=0$ in $\Omega$ with $u|_{\partial\Omega}=f$. The existence and the uniqueness of solutions is ensured if, for instance, we assume $0$ to be not a Dirichlet $L^2(\Omega)$-eigenvalue of $\mathcal{L}_{A,q}$. 
According to the choice of the sets $F$ and $B$, we can mainly distinguish two classes of IBVPs:

\begin{itemize}
\item Full data. When $F=B=\partial\Omega$.
\item Partial data. Either, $F\neq \partial\Omega$ or $B\neq  \partial \Omega$ are nonempty sets.
\end{itemize}

The DN map associated to full data cases has a gauge invariance \cite{Sun}. In fact, if $\varphi\in C^1(\overline{\Omega})$ is a real-valued function with $\varphi |_{\partial\Omega}=0$, then one has $\Lambda^{\partial\Omega\rightarrow \partial\Omega}_{A,q}= \Lambda^{\partial\Omega\rightarrow \partial\Omega}_{A+\nabla \varphi, q}$. It shows that the DN map does not distinguish perturbations in gradient forms of the magnetic potentials. Hence we only  expect to recover $dA$ and $q$, which is what actually happens, see for instance \cite{Sun}, \cite{Tz}.\\

The framework of our work is contained in the category of IBVPs with partial data. Throughout this manuscript, we consider $B=\partial \Omega$ and $F$ being an open neighborhood of the illuminated boundary region of $\partial \Omega$ defined by
\begin{equation}\label{cset}
 \partial \Omega_{-,0}(x_0):= \left \{    x\in \partial\Omega  \; : \; \left \langle x-x_0, \nu(x) \right \rangle \leq 0 \right \},
\end{equation}
where $x_0\in \mathbb{R}^n\setminus ch(\overline{\Omega})$. Here $ch(\overline{\Omega})$ stands for the convex hull of $\overline{\Omega}$. Notice that if $\Omega$ is  strictly convex then $F $ could be arbitrarily small. We also consider the shadowed boundary region:
\begin{equation}\label{cset1}
 \partial \Omega_{+,0}(x_0):= \left \{    x\in \partial\Omega  \; : \; \left \langle x-x_0, \nu(x) \right \rangle \geq 0 \right \}.
\end{equation}

To define an appropriate DN map we introduce a boundary cutoff function $\chi_F$ supported on $F$ such that it equals to $1$ on $\partial \Omega_{-,0}(x_0)$. Thus, the partial DN map $\Lambda_{A,q}^{\partial\Omega\rightarrow F}:=\Lambda_{A,q}^{\sharp}$ is defined as \begin{equation}\label{DNmapsz1}
\begin{matrix}

\Lambda_{A,q}^{\sharp}:&H^{1/2}(\partial \Omega)  &\rightarrow&H^{-\frac{1}{2}}(\partial \Omega) \\ 
 &f& \mapsto&\chi_{F}\, \Lambda^{\partial\Omega\rightarrow \partial\Omega}_{A,q} f.
\end{matrix}
\end{equation}

Our first result states that the magnetic field and the electric potential are wholly determined (in $\Omega$) by $\Lambda_{A, q}^{\sharp}$, like to the full data case.
\begin{thm}\label{unique_our_result}
Let $A_j\in C^{1}(\Omega; \mathbb{C}^n)$ and $q_j\in L^{\infty}(\Omega)$ for $j=1,2$. Assume that $0$ is not a Dirichlet eigenvalue in $L^2(\Omega)$ of $\mathcal{L}_{A_j, q_j}$. If
\[
\Lambda_{A_1, q_1}^{\sharp}f= \Lambda_{A_2, q_2}^{\sharp}f\quad \text{for all}\; f\in H^{1/2}(\partial\Omega),
\]
then $dA_1=dA_2$ and $q_1=q_2$ in $\Omega$.
\end{thm}
This result is not really new. It was proved in \cite{DSFKSjU}-\cite{Ch} by assuming $A$ to be a $C^2$ vector field and $q$ bounded, and when $A\in W^{1,n}\cap L^\infty$ and $q\in L^n$ in \cite{KS}. It is also the aim of this work to provide the corresponding quantitative versions of Theorem \ref{unique_our_result}. To derive stability results, one needs \textit{a priori} bounds on the potentials to control their high oscillations. According to \cite[Proposition 3.6]{SiW}, see also \cite[Lemma 1.1]{FRo}, it is known that the characteristic function of $\Omega$ (denoted by $\chi_\Omega$) belongs to $H^{\sigma}(\mathbb{R}^n)$ with $\sigma\in (0,1/2)$, whenever $\partial \Omega$ is smooth enough, as in our case. Motivated by this fact, we consider the class of admissible potentials. 
\begin{deft} \label{adclmp} Given $M>0$ and $\sigma\in\left( 0,1/2\right)$, we define the {\bf class of admissible magnetic potentials} $\mathscr{A}(\Omega, M, \sigma)$ by
\[
\mathscr{A}(\Omega, M, \sigma)= \left \{ W \in  C^{1+\sigma}(\mathbb{R}^n; \mathbb{C}^n) : \supp{W}\subset \overline{\Omega},  \left \| W \right \|_{C^{1+\sigma}} \leq M \right \}.
\]
\end{deft}

\begin{deft} \label{adclmp1} Given $M>0$ and $\sigma\in\left( 0,1/2\right)$, we define the {\bf class of admissible electric potentials} $\mathscr{Q}(\Omega, M, \sigma)$ by 
\[
\mathscr{Q}(\Omega, M, \sigma)= \left \{ V \in  L^\infty \cap H^\sigma (\mathbb{R}^n; \mathbb{C})  : \supp{V}\subset \overline{\Omega},  \left \| V \right \|_{L^\infty} + \left \|  V \right \|_{H^\sigma}\leq M \right \}.
\]
\end{deft}

 For a function $h:\Omega\rightarrow \mathbb{C}$ (or $\mathbb{C}^n$), we denote by $\chi_{\Omega} h$ its extension by zero out of $\Omega$. Throughout this paper, we write $a\lesssim b$ whenever $a$ and $b$ are non-negative quantities that satisfy $a\leq C b$ for a constant $C > 0$ only depending on $\Omega$ and a-priori assumptions on the potentials. 
 
\begin{thm}\label{SMP}
Let $M>0$, $\sigma\in (0,1/2)$ and $j=1,2$. Assume that $0$ is not a Dirichlet eigenvalue in $L^2(\Omega)$ of $\mathcal{L}_{A_j, q_j}$. Then there exists $C>0$ (depending on $n, \Omega, M, \sigma$) such that the following estimate 
\[
\left \|d  A_1-d A_2   \right \|_{L^2(\Omega)} \leq C \left | \log \left | \log  \left \|\Lambda^{\sharp}_{1}- \Lambda^{\sharp}_{2}    \right \| \right |   \right |^{-\frac{\sigma}{3(1+\sigma)}}
\]
holds for all $q_j\in L^{\infty}(\Omega)$ and for all $\chi_{\Omega}A_j\in \mathscr{A}(\Omega, M, \sigma)$ satisfying $A_1=A_2$ and $\partial_\nu A_1=\partial_\nu A_2$ on $\partial\Omega$.
\end{thm}

 \begin{thm}\label{SEP}
Let $M>0$, $\sigma\in (0,1/2)$ and $j=1,2$. Assume that $0$ is not a Dirichlet eigenvalue in $L^2(\Omega)$ of $\mathcal{L}_{A_j, q_j}$. Then there exist $C>0$ (depending on $n, \Omega, M, \sigma$) such that the following estimate
\[
\left \| q_1-q_2  \right \|_{L^2(\Omega)} \leq C \left | \log \left |  \log \left | \log  \left \|\Lambda^{\sharp}_{1}- \Lambda^{\sharp}_{2}    \right \| \right |  \right | \right |^{-\frac{\sigma}{3(\sigma+1)}} 
\]
holds for all $\chi_{\Omega}q_j \in \mathscr{Q}(\Omega, M,\sigma)$ and for all $\chi_{\Omega}A_j \in \mathscr{A}(\Omega, M, \sigma)$ satisfying $A_1=A_2$ and $\partial_\nu A_1=\partial_\nu A_2$ on $\partial\Omega$. 
\end{thm}
 The notation $\left \| \, \cdot\, \right \|$ in above Theorem \ref{SMP} and Theorem \ref{SEP} stand for $\left \| \, \cdot\,  \right \|_{H^{1/2}(\partial\Omega)\rightarrow H^{-1/2}(\partial\Omega)}$, which in turn is defined in a standard way by 
\begin{equation}\label{partialnorm}
\left \| \Lambda^{\sharp}_{A,q} \right \|_{H^{1/2}(\partial\Omega)\rightarrow H^{-1/2}(\partial\Omega)} :=\underset{  \left \| f \right \|_{H^{1/2}(\partial\Omega)}=1}{\sup}  \left \| \chi_F \Lambda_{A,q} f   \right \|_{H^{-1/2}(\partial\Omega)}.
\end{equation}

We now describe some earlier results. In the absence of a magnetic potential, our results are closely connected with the widely studied Calder\'on's problem \cite{C}. It consists of recovering the interior conductivity of a body from electrical measurements on its whole boundary.  A weak version of this problem was first studied in \cite{BU}, where the authors proved that it is enough to take measurements on large subsets of the boundary for recovering the electric potential. Since the IBVPs are high ill-posed, it is not surprising to get moduli of logarithmic type continuities when obtaining stability estimates \cite{A}, \cite{HW}. These moduli are the best one expect to get, at least for full data cases and by the method of complex geometric solutions, as it was showed in \cite{Ma}. We have not attempted to be exhaustive in references related to IBVPs in other settings with incomplete information in the absence of magnetic potentials, and closely related inverse problems like reconstruction methods. We recommend \cite{KS1} for a more comprehensive bibliography and for a more gentle introduction to these issues. The first uniqueness result of a full data case was obtained for small magnetic potentials \cite{Sun}. This condition was removed in \cite{NSU} by assuming $C^2$ compactly supported magnetic potentials. See also \cite{DSFKSU} for uniqueness on manifolds. Uniqueness for full data was improved in \cite{KU}, only assuming bounded potentials. In this case, stability estimates were derived in \cite{CP}. \\
  
The proof of Theorem \ref{unique_our_result} is an immediate consequence of Theorem \ref{SMP} and Theorem \ref{SEP}. Theorem \ref{unique_our_result} was first proved in \cite{KSU} when $A\equiv 0$ and stability estimates were obtained in consecutive works in \cite{CDSFR} and \cite{CDSFR1} by proving proper bounds for the Radon transform and the attenuated geodesic ray transform. In the presence of a magnetic potential, Theorem \ref{unique_our_result} was proved in \cite{DSFKSjU}, \cite{KS}, and \cite{Ch}, under smoothness assumptions on the potentials. Our stability estimates in Theorem \ref{SMP} and Theorem \ref{SEP} are the quantification of the corresponding uniqueness results derived in \cite{DSFKSjU}, \cite{KS}, and \cite{Ch}. \\

The proofs of Theorem \ref{SMP} and Theorem \ref{SEP} will be carried out by deriving first an integral identity relating the partial boundary data, i.e., the partial DN maps, with the unknown magnetic and electric potentials. After that, we construct complex geometric optics (CGO) solutions $u\in H^1(\Omega)$ to $\mathcal{L}_{A,q}u=0$ in $\Omega$, through a suitable Carleman estimate for a conjugate version of $\mathcal{L}_{A,q}$. To obtain information about the potentials, we insert the CGO solutions into the first step's integral identity. Now we have to deal with boundary terms from the shadowed boundary region $ \partial \Omega_{+,0}(x_0)$ defined by \eqref{cset1}. We use another Carleman estimate with boundary terms, which, roughly speaking, gives controllability of the boundary terms coming from the shadowed region $ \partial \Omega_{+,0}(x_0)$  by similar terms from the illuminated part  $ \partial \Omega_{-,0}(x_0)$. This step might be set aside by using another CGO solutions vanishing on the shadow region of the boundary as in \cite{Ch}, but in counterpart, we shall need to increase the regularity on the magnetic potentials (slightly bigger than $C^2$). After introducing coordinates from $\Omega$ to a compact subset of $\mathbb{C}\times S^{n-2}$, we get information on the attenuated geodesic ray transform of certain functions involving the potentials. To decouple the information from the previous step, we establish estimates for the attenuated geodesic ray transform, which are valid for small real attenuations, see Theorem \ref{stability_estimates_st_ra_tr}. To transfer the information from small attenuations to the  whole real line, we have to add one logarithm in our estimates. This fact is closely connected with the quantification of Fourier transform's unique continuation, see Lemma \ref{lemma_estension_caro_ruiz_DSF}. In most of our computations, we will only need the magnetic potentials of class $C^1$. The extra regularity $C^{1+\sigma}$ is only needed to use the Fourier transform's unique continuation, see Lemma \ref{lemma_estension_caro_ruiz_DSF}. The conditions $A_1 = A_2$ and $\partial_\nu A_1=\partial_\nu A_2$ on $\partial\Omega$ are needed to extend $A_1-A_2$ to a slightly larger open set than $\Omega$ while preserving the $C^{1+\sigma}$ regularity. The situation is a bit different when proving Theorem \ref{SEP}. In this part, we take advantage of the DN map's gauge invariance to use the already established stability estimate for the magnetic fields, utilizing a Hodge decomposition for $A_1-A_2$ derived in \cite{Tz}. This step involves two logarithms coming from Theorem \ref{SMP}, and once again, an extra logarithm has to be added to extend the Fourier transform information from small attenuations to $\mathbb{R}$. \\

On the other hand, when a fixed direction in $S^{n-1}$ determines the illuminated boundary region, uniqueness was obtained  \cite{BU} to the case $A\equiv 0$, and in the presence of a magnetic potential in \cite{Tz}. The author in \cite{Tz} also obtained stability estimates of logarithmic type. They considered partial boundary information in open sets slightly larger than the half of $\partial\Omega$. Note that in our case, the boundary information can be taken in arbitrary small subsets of $\partial\Omega$ when $\Omega$ is convex. This geometric feature makes a huge difference in studying both IBVPs. One of the most noteworthy of which are the coordinates we work. The CGO solutions in the presence of a magnetic potential involve an exponential term, whose exponent satisfies a $\overline{\partial}$-equation. This equation can be solved by using the so-called Cauchy transform. The exponential term must be adequately removed from our estimates while keeping track of the constants' dependence on all intermediate estimates. In \cite{KU}, \cite{Sa1}, and \cite{Sun}, the coordinates introduced in $\Omega$ are globally defined, so it is possible to remove the exponential term by using the asymptotic behavior at infinity of the Cauchy transform. In particular, one can take limits when the spatial variable $ | x | \to \infty$ to remove the exponential term from the intermediate estimates. This argument does not work in our case because our approach uses geodesic coordinates (natural coordinates of the attenuated geodesic ray transform), which induces a variable living in the complex plane $\mathbb{C}$. We cannot use the previous asymptotic argument since we might intersect some branches where the complex logarithm function could not be well defined when taking the complex variable $| z | \to \infty$. To overcome this difficulty, in Lemma \ref{quantification_complex_argument} we derive a quantitative version of holomorphic extensions into the complex unit ball of functions initially defined in the complex unit sphere.\\

This paper is organized as follows. In Section \ref{preli}, we introduce the attenuated geodesic ray transform and its main properties. Section  \ref{c_g_o_sulzlo}  is devoted to constructing CGO solutions to the magnetic Schr\"odinger equation by using geodesic coordinates. In Section \ref{sectio_four_st} and Section \ref{sta_electri_pot}, we prove Theorem \ref{SMP} and Theorem \ref{SEP}, respectively. Finally, the Appendix is entirely dedicated to explaining how to remove the exponential term satisfactorily when proving stability estimates for the magnetic potentials; see Theorem \ref{remov_esti_phi_1}.



\section{Attenuated geodesic ray transform on hemispheres} \label{preli}

%

In this section, we prove the attenuated geodesic ray transform's continuity properties defined on a slightly smaller open subset than the upper half hemisphere. For $\hbar \in [0,1)$, let first define
\[
S^{n-1}_{>\hbar}= \left\{ x\in S^{n-1}\, : \, x_n>\hbar\right\}\subset \mathbb{R}^n. 
\]
Now consider $0<\beta^\prime<\beta <1$ and $S^{n-1}_{>\beta^\prime}$ provided with the metric induced by the canonical Euclidean metric. The set of boundary inward-pointing unit vectors to the unit sphere bundle 
 \[
S(S^{n-1}_{>\beta^\prime}) = \left\{ (y, \eta)\in TS^{n-1}_{>\beta^\prime}: |\eta|=1\right\}
\]
is given by 
\[
\partial_+ S (S^{n-1}_{>\beta^\prime})=\left\{ (y, \eta)\in S(S^{n-1}_{>\beta^\prime}) :  y \in \partial S^{n-1}_{>\beta^\prime}, \ \left \langle \eta, \nu(y) \right \rangle <0 \right\},
\]
where $\nu$ denotes the outer unit normal vector to $\partial S^{n-1}_{>\beta^\prime}$. Let $H$ be a complex-valued function defined on $S (S^{n-1}_{>\beta^\prime})$. The geodesic ray transform with real attenuation $-\lambda<0$ of $H$ at $(y, \eta)\in \partial_+ S(S^{n-1}_{>\beta^\prime})$ is defined by
\begin{equation}\label{defi_atte_geod_ray_trans}
(T_\lambda H)(y, \eta)= \int_0^{\tau(y, \eta)} e^{-\lambda \theta} H(\gamma_{y, \eta}(\theta), \dot{\gamma}_{y, \eta}(\theta)) \, d\theta,
\end{equation}
where $\tau(y, \eta)$ denotes the first arrival time to $\partial S^{n-1}_{>\beta^\prime}$ of the unit speed geodesic $\gamma_{y, \eta}$ starting at $y$ with initial velocity $\eta$. The upper dot $\dot{}:= \frac{d}{d\theta}$ represents the derivative of the variable $\theta$. 
\begin{rem} \label{rema_finite_measure}
Following \cite[Section 2.1]{CDSFR1}, we consider the following measure on $\partial_+S(S^{n-1}_{>\beta^\prime})$  
\[
d\mu (y, \eta) =| \left \langle \eta, \nu(y)\right \rangle| dy \, d\eta
\]
and the corresponding $L^2$-norm by
\begin{align*}
\left\| F \right\|_{L^2(\partial_+S(S^{n-1}_{>\beta^\prime}))}^2&= \int_{\partial_+S(S^{n-1}_{>\beta^\prime})} |F(y, \eta)|^2 d\mu(y, \eta)\\
& = \int_{\partial S^{n-1}_{>\beta^\prime}} \int_{y^\perp \cap S^{n-1}_{>0}}  |F(y, \eta)|^2|  \left \langle \eta, \nu(y) \right \rangle | d\eta \, dy. 
\end{align*}
\end{rem}
Using Santal\'{o}'s formula, see for instance \cite[Lemma A.8]{DPSU}:
\begin{equation}\label{Santalo_formula}
\int_{S(S^{n-1}_{>\beta^\prime})} F(x, \xi) dx\, d\xi= \int_{\partial_+ S (S^{n-1}_{>\beta^\prime})} \int_{0}^{\tau(y, \eta)} F \left(\gamma_{y, \eta}(\theta), \dot{\gamma}_{y, \eta}(\theta) \right) d\theta\, d\mu(y, \eta)
\end{equation}
one can deduce the continuity of $T_\lambda$ in Sobolev spaces. In particular, if $F$ only depends on the variable $x$, then
\[
\int_{S^{n-1}_{>\beta^\prime}} F(x) dx= \dfrac{1}{|S^{n-2}|}\int_{\partial_+ S (S^{n-1}_{>\beta^\prime})} \int_{0}^{\tau(y, \eta)} F \left(\gamma_{y, \eta}(\theta) \right) d\theta\, d\mu(y, \eta).
\]

\begin{lem}\label{continuity:attenuated_ray_transform_1}
Let $\lambda\in \mathbb{R}$ and $\sigma \in [0,1]$. The attenuated geodesic ray transform $T_\lambda$ defined by \eqref{defi_atte_geod_ray_trans} is a bounded operator from $H^\sigma (S(S^{n-1}_{>\beta^\prime}))$ to $H^\sigma(\partial_+S(S^{n-1}_{>\beta^\prime}))$.
\end{lem}
\begin{proof}
We only prove the result for $\sigma=0$ and $\sigma=1$. The result for intermediate values will follow by interpolation. By a density argument, it is enough to prove the result for every $H\in C^{\infty}_c(S(S^{n-1}_{>\beta^\prime}))$, and indeed we shall assume this regularity through the computations in this proof. \\

\noindent {\it Case $\sigma=0$}. Combining Cauchy-Schwarz's inequality with a direct application of Santal\'{o}'s formula \eqref{Santalo_formula} to $|H|^2$, we obtain 
\begin{align*}
\left\| T_\lambda H \right\|_{L^2(\partial_+ S(S^{n-1}_{>\beta^\prime}))}^2&=\int_{\partial_+ S(S^{n-1}_{>\beta^\prime})}| (T_\lambda H)(y, \eta)|^2 d\mu(y, \eta) \\
& \leq \pi \, e^{2|\lambda|\pi} \int_{\partial_+ S(S^{n-1}_{>\beta^\prime})}\int_0^{\tau(y, \eta)} |H (\gamma_{y, \eta}(\theta),  \dot{\gamma}_{y, \eta}(\theta) )|^2 d\theta\, d\mu(y, \eta)\\
& = \pi \, e^{2|\lambda|\pi}\int_{S(S^{n-1}_{>\beta^\prime})} |H(x, \xi)|^2 dx\, d\xi \\
& = \pi \, e^{2|\lambda|\pi} \left\| H\right\|_{L^2(S(S^{n-1}_{>\beta^\prime}))}^2,
\end{align*}
which immediately implies the $L^2$-continuity of $T_\lambda$. \\

\noindent {\it Case $\sigma=1$}. In this case, it is enough to relate the differential maps of $T_\lambda H$ and $H$ in their corresponding domains, and then use again Santal\'{o}'s formula \eqref{Santalo_formula} as in case $\sigma=0$. Consider $(y, \eta)\in \partial_+S(S^{n-1}_{>\beta^\prime})$. Recall that $dT_\lambda H(y, \eta)$ is a linear map from $T_{(y, \eta)} \,\partial_+S(S^{n-1}_{>\beta^\prime})$ to $\mathbb{C}$. To obtain a suitable bound for $d\,T_\lambda H (y, \eta)$, we consider $(y^\prime, \eta^\prime)\in T_{(y, \eta)} \,\partial_+S(S^{n-1}_{>\beta^\prime})$ and for $\epsilon_0>0$ small enough, a smooth curve $\Gamma: (-\epsilon_0, \epsilon_0)\to \partial_+ S(S^{n-1}_{>\beta^\prime})$ so that $\Gamma(0)=(y, \eta)$ and $(\frac{d}{ds} \Gamma)_{|_{s=0}}=(y^\prime, \eta^\prime)$. Writing $\Gamma(s)=(\Gamma_1(s), \Gamma_2(s))$, these conditions read as
\[
\Gamma_1(0)=y, \quad   \Gamma_2(0)=\eta, \quad  \left(\frac{d}{ds} \Gamma_1\right)_{|_{s=0}}=  y^\prime, \quad \left(\frac{d}{ds} \Gamma_2\right)_{|_{s=0}}=  \eta^\prime. 
\]
By definition, we have in $\mathbb{C}$: 
\[
\left( d T_\lambda H (y, \eta)\right) (y^\prime, \eta^\prime)= \left(\dfrac{d}{ds} (T_\lambda\, H) (\Gamma(s))\right)_{|_{s=0}}.
\]
A direct application of the chain rule yields
\begin{align*}
& \dfrac{d}{ds} (T_\lambda\, H) (\Gamma(s))\\
 & = \int_{0}^{\tau(\Gamma_1(s), \Gamma_2(s))} e^{-\lambda \theta} \dfrac{d}{ds}  \left(H \left(\gamma_{\Gamma_1(s), \Gamma_2(s)}(\theta), \dot{\gamma}_{\Gamma_1(s), \Gamma_2(s)} (\theta)  \right)\right) d\theta \\
&\qquad  + e^{-\lambda \,  \tau (\Gamma_1(s), \Gamma_2(s))}   \frac{d}{ds} \left( \tau (\Gamma_1(s), \Gamma_2(s)) \right)\\
&\qquad \qquad  \times H\left( \gamma_{  \Gamma_1(s), \Gamma_2(s)}(  \tau (\Gamma_1(s), \Gamma_2(s))),  \dot{\gamma}_{ \Gamma_1(s), \Gamma_2(s)}(  \tau (\Gamma_1(s), \Gamma_2(s))) \right).
\end{align*}
Since $H$ is compactly supported on $S(S^{n-1}_{>\beta^\prime})$, it follows that the second term on the right vanishes when evaluating at $s=0$. Hence 
\[
\left( d T_\lambda H (y, \eta)\right) (y^\prime, \eta^\prime)= \int_{0}^{\tau(y, \eta)} e^{-\lambda \theta} \left(  \dfrac{d}{ds} H \left(\gamma_{\Gamma_1(s), \Gamma_2(s)}(\theta), \dot{\gamma}_{\Gamma_1(s), \Gamma_2(s)} (\theta) \right) \right)_{|_{s=0}}d\theta. 
\]
To compute the derivative at $s=0$ in the right-hand side, we introduce the geodesic flow. For each $\theta\in (-\epsilon_0, \epsilon_0)$, consider the map $\Theta_\theta: S(S^{n-1}_{>\beta^\prime})\to S(S^{n-1}_{>\beta^\prime})$ defined by
\[
\Theta_\theta (x, \xi)= (\gamma_{x, \xi}(\theta), \dot{\gamma}_{x, \xi}(\theta) ).
\]
Note that $\Theta_0$ is the identity operator in $S(S^{n-1}_{>\beta^\prime})$, $\Theta_{\theta_1+ \theta_2}=\Theta_{\theta_1}\circ \Theta_{\theta_2}$ for every $\theta_1, \theta_2$ small enough. Moreover, the map $\Theta_{\theta}$ is a diffeomorphism for each $\theta\in (-\epsilon_0, \epsilon_0)$. It allows us to define the smooth map $\Theta:  (-\epsilon_0, \epsilon_0)\times S(S^{n-1}_{>\beta^\prime})\to S(S^{n-1}_{>\beta^\prime})$ by 
\[
\Theta (\theta; (x, \xi))= \Theta_\theta(x, \xi). 
\]
In geometric language, this map is usually called the geodesic flow associated to $S^{n-1}_{>\beta^\prime}$.  We refer the reader to \cite[Chapter 1]{GPa} and \cite[Notes 8]{IPRGJYU} for a detailed treatment of the geodesic flow's regularity properties in Riemannian manifolds. Fixing $\theta \in (0, \tau(y, \eta))$, $s\mapsto \Theta_\theta(\Gamma_1(s), \Gamma_2(s))$ is a curve in $S(S^{n-1}_{>\beta^\prime})$ with initial condition $\Theta_\theta(y, \eta)$ and initial velocity $\left( \frac{d}{ds}\Theta_\theta(\Gamma_1(s), \Gamma_2(s))\right)_{|_{s=0}}$. By definition we have
\[
\left( \frac{d}{ds}\Theta_\theta(\Gamma_1(s), \Gamma_2(s))\right)_{|_{s=0}}= (d\Theta_\theta (y, \eta))(y^\prime, \eta^\prime).
\]
Adding up all previous facts, we get
\begin{align*}
&\left(  \dfrac{d}{ds} H \left(\gamma_{\Gamma_1(s), \Gamma_2(s)}(\theta), \dot{\gamma}_{\Gamma_1(s), \Gamma_2(s)} (\theta) \right) \right)_{|_{s=0}}\\
&=\left(  \dfrac{d}{ds} H \left(\Theta_\theta (\Gamma_1(s), \Gamma_2(s)) \right)\right)_{|_{s=0}} \\
&=\left( dH(\gamma_{y, \eta}(\theta), \dot{\gamma}_{y, \eta}(\theta) ) \right) \left( (d\Theta_\theta (y, \eta))(y^\prime, \eta^\prime) \right).
\end{align*}
For our purpose, it is enough to know that $d \Theta_\theta (y, \eta)$ remains uniformly bounded with universal constant independent of $\theta$ and $(y, \eta)$. Indeed, by \cite[Lemma 1.40]{GPa}, we have
\[
d \Theta_\theta (y, \eta)(y^\prime, \eta^\prime)= J_{(y^\prime, \eta^\prime)}(\theta)+ D_\theta J_{(y^\prime, \eta^\prime)}(\theta),
\]
where $J_{(y^\prime, \eta^\prime)}$ is the Jacobi field along the geodesic $\theta \mapsto \gamma_{y, \eta}(\theta)$ with initial conditions $J_{(y^\prime, \eta^\prime)}(0)= (y^\prime, \eta^\prime)_h$ and $D_\theta J_{(y^\prime, \eta^\prime)}(0)= (y^\prime, \eta^\prime)_v$. Here $D_{\theta}$, $(y^\prime, \eta^\prime)_h$ and $(y^\prime, \eta^\prime)_v$ stand for the covariant derivative with respect to $\theta$, and the horizontal and vertical decomposition of $(y^\prime, \eta^\prime)$, respectively. Combining the above computations, we get the estimate
\[
|(dT_\lambda H)(y, \eta)|^2 \leq C \pi \, e^{2|\lambda|\pi} \int_{0}^{\tau(y, \eta)} |dH (\gamma_{y, \eta}(\theta), \dot{\gamma}_{y, \eta}(\theta))|^2 d\theta.
\]
Finally, integrating both sides in $\partial_+ S(S^{n-1}_{>\beta^\prime})$ and using Santal\'{o}'s formula \eqref{Santalo_formula}, the case $\sigma=1$ follows immediately. The proof is completed. 
\end{proof}





From now on, we restrict our study of $T_\lambda$ to a particular family of smooth functions in $S(S^{n-1}_{>\beta^\prime})$. Denote by $p=(x, \xi)$ the elements of $S(S^{n-1}_{>\beta^\prime})$, where $x=\pi(p)$ and $\xi$  its corresponding unit tangent vector. Given a function $f\in C^\infty(S^{n-1}_{>\beta^\prime})$ and  a  1-form field $ \alpha\in C^{\infty}(S^{n-1}_{>\beta^\prime}; \Lambda^1(S^{n-1}_{>\beta^\prime}))$, we can define the affine function (affine in $\xi$ for each fixed $x$) on $S(S^{n-1}_{>\beta^\prime})$ by 
\begin{equation}\label{affine_function}
H(p)= f(x)+ \alpha(x)(\xi).
\end{equation}
The operator $T_\lambda$ has a non-empty kernel in this subset of functions. Indeed,  $T_\lambda(-\lambda \ss + d\ss)=0$ for all smooth function $\ss\in C^\infty(S^{n-1}_{>\beta^\prime})$ with $\ss|_{\partial S^{n-1}_{>\beta^\prime}}=0$. In particular, Lemma \ref{continuity:attenuated_ray_transform_1} gives us bounds from above for $T_\lambda$. However, it is not obvious at all how to get similar bounds from below. It can be addressed by using the normal operator $T_\lambda^*\,T_\lambda$, which is an elliptic pseudodifferential operator of order $-1$ acting on solenoidal pairs, see for instance \cite[Proposition 1]{HS}, \cite[Proposition 4.1]{FSU} and \cite[Proposition 3.1]{YMA}. Here $T_\lambda^*$ stands for the adjoint of $T_\lambda$.

\begin{deft}\label{deft_solenoidal_partz}
We say that a pair $(f, \upalpha)\in L^2(S^{n-1}_{>\beta^\prime})\times L^2( \Lambda^1S^{n-1}_{>\beta^\prime})$ ($f$ is a function and $\upalpha$ an $1$-form on $S^{n-1}_{>\beta^\prime}$) is a {\bf solenoidal pair} if $\delta_{S^{n-1}}\upalpha=0$. Here $S^{n-1}$ is equipped with the canonical metric of $\mathbb{R}^n$ and $\delta_{S^{n-1}}$ denotes the divergence operator on $S^{n-1}$.

\end{deft}

\begin{lem}
\label{abstract estimate}
Let $X$, $Y$, and $Z$ be Banach spaces and let $A_\lambda : X\to Y$ be a family of injective linear operators depending continuously on $\lambda\in \R$ with respect to the norm topology on bounded linear operators. Suppose there is a family of compact operators $K_\lambda : X\to Z$, again depending continuously on $\lambda\in \R$ with respect to the norm topology on bounded linear operators. Suppose that there exists $\lambda_0>0$ such that for all $|\lambda|\leq \lambda_0$ we have the uniform estimate
$$\|u\|_X\leq C(\|A_\lambda u\|_{Y} + \|K_\lambda u\|_Z)$$
then there exists $C'\in\R$ such that
$$\|u\|_X \leq C' \, \|A_\lambda u\|_{Y}$$
for all $|\lambda|\leq \lambda_0$.
\end{lem}
\begin{proof}
Suppose by contradiction that there is a sequence $\lambda_j$ and $u_j$ with $\|u_j\|_{X} = 1$ such that $A_{\lambda_j} u_j \to 0$. Without loss of generality we may assume that $\lambda_j \to \hat\lambda$ and $\{K_{\hat\lambda}u_j\}$ is a Cauchy sequence. Taking a limiting argument we have that
\begin{eqnarray}
\label{hat estimate}
\|u\|_X \leq C (\|A_{\hat\lambda}u\|_Y + \|K_{\hat\lambda} u \|_Z).
\end{eqnarray}
Using the triangle inequality, we can deduce that $\|A_{\hat\lambda} u _j \|_Y \to 0$.
Insert $u_j - u_k$ into estimate \eqref{hat estimate} we get
$$ \|u_j - u_k\|_X \leq C(\|A_{\hat\lambda}(u_j - u_k)\|_Y + \|K_{\hat\lambda} (u_j - u_k)\|_{Z}).$$
This means that $\{u_j\}$ is a Cauchy sequence converging to some $\hat u$. We then have that $A_{\hat \lambda} \hat u = 0$ contradicting injectivity. 
\end{proof}
\begin{thm}\label{stability_estimates_st_ra_tr}
Let $0<\beta'<\beta<1$ be given by Lemma \ref{g_e_c_o_s}. Then there exist $\lambda_0>0$ and $C>0$ such that for all $0\leq \lambda \leq \lambda_0$ one has the following estimate
\begin{equation}\label{shifted_version}
\left\| f \right\|_{H^{-1}(S^{n-1}_{>\beta})} + \left\| \upalpha \right\|_{H^{-1}(\Lambda^1S^{n-1}_{>\beta})}  \leq C \left\| T_\lambda^*\,T_\lambda [f, \upalpha] \right\|_{{L^{2}( S(S^{n-1}_{>\beta'}))}}
\end{equation}
for all solenoidal pair $(f, \upalpha)$ which are compactly supported on $S^{n-1}_{>\beta}$. 
\end{thm}

\begin{rem} 
We mention the recent work \cite{YMA}, where the author proved a similar version of \eqref{shifted_version} when $T_\lambda^*\,T_\lambda$ is acting on roughly $1$-forms plus symmetric $2$-tensors in a simple manifold, instead of functions plus $1$-forms as our case, see \eqref{affine_function}. Our case is a particular case of \cite[Proposition 3.1]{YMA} because any function $f\in C^\infty(S^{n-1}_{>\beta^\prime})$ can be identified with a symmetric $2$-tensor by $f(x)\widetilde{g}_x(\xi, \xi)$, where $\widetilde{g}$ denotes the metric on $S^{n-1}_{>\beta^\prime}$. Under this identification, we could use \cite[Proposition 3.1]{YMA} to deduce Theorem \ref{stability_estimates_st_ra_tr}. However, for the convenience of the reader and for the sake of completeness, we give a proof of Theorem \ref{stability_estimates_st_ra_tr}. 
\end{rem}

\begin{proof}
Denote by $T_\lambda^*$ the adjoint operator to $T_\lambda$. The normal operator  $T_\lambda^* T_\lambda$ is then given by the following matrix of operators
\begin{eqnarray}
\label{normal op}
T_\lambda^*T_\lambda[f,\upalpha] =  {\begin{bmatrix}
 (T^0_\lambda)^* T_\lambda^0 & (T^0_\lambda)^* T^1_\lambda \\
   (T^1_\lambda)^* T_\lambda^0&  (T_\lambda^1)^* T_\lambda^1  \\
     \end{bmatrix}} \begin{bmatrix} f\\ \upalpha\end{bmatrix}
\end{eqnarray}
where $T_\lambda^j$ denotes the attenuated geodesic transforms with attenuation $e^{-\lambda t}$ on $S_{>\beta'}^{n-1}$ acting on symmetric $j$-tensors. We would like to compute the principal symbol of the operator 
\begin{eqnarray}
\label{Elambda}
 E_\lambda:=  T_\lambda^* T_\lambda  +  \begin{bmatrix} 0&0\\0& d \langle \epsilon D\rangle^{-3/2} \langle\epsilon D\rangle^{-3/2}\delta  \\ \end{bmatrix}
\end{eqnarray}
where $\langle \epsilon D\rangle^{r}$ a self-adjoint $\Psi$DO on $S^{n-1}$ with principal symbol $(1 + \epsilon^2|\xi|^2)^{r/2}$. Observe that by treating $\epsilon>0$ as a semiclassical parameter, we have the following result.
\begin{lem}
\label{bijection of jap bracket}
There exists a fixed $\epsilon_0>0$ such that for all $\epsilon >\epsilon_0$ small enough, the operator $\langle \epsilon D\rangle^r$ is a bijection from $H^{s}(S^{n-1}) \to H^{s-r}(S^{n-1})$ (though the norm of the bijection will not be uniform in $\epsilon$.)
\end{lem}
\begin{proof}
Standard semiclassical $\Psi$DO calculus states $\langle \epsilon D\rangle^{-r} \langle \epsilon D\rangle^{r} = I + \epsilon{\rm Op}_\epsilon (a)$ for some semiclassical symbol $a$ of order $-1$. By taking $\epsilon$ small enough, we get an inverse by Neumann series on semiclassical Sobolev spaces parametrized by $\epsilon$. Therefore, for each fixed $\epsilon$ sufficiently small we have that $\langle \epsilon D\rangle^{r}$ is a bijection between $H^{s}(S^{n-1}) \to H^{s-r}(S^{n-1})$.
\end{proof}

 Let $g$ be the metric on $S^{n-1}$ and denote by $m_g$ the operator $m_g : f\mapsto f g$.  The principal symbol of $E_\lambda$ can then be computed by observing that $T_\lambda^0 = T_\lambda^2 m_g$ and use the formula derived in proof of  \cite[Proposition 3.1]{YMA}. In normal coordinates centered at $x$, they are given by
\begin{equation}\label{integrals}
\begin{aligned}
& \sigma ((T_\lambda^0)^* T_{\lambda}^0)(x,\xi) f\\
&\quad  = 2\pi\varphi|\xi|^{-1}\int_{\omega \in S_x S^{n-1}_{>\beta'}} \delta(|\xi|^{-1} \xi(\omega))e^{-2\lambda \tau(x,-\omega)} d\omega\\
&\sigma ((T_\lambda^1)^* T_{\lambda}^0)(x,\xi) f \\
&\quad = 2\pi\varphi|\xi|^{-1}\left( \int_{\omega \in S_x S^{n-1}_{>\beta'}}  \omega^k \delta(|\xi|^{-1} \xi(\omega)) e^{-2\lambda \tau(x,-\omega)}d\omega \right) \delta_{j,k} dx^j\\
&\sigma ((T_\lambda^0)^* T_{\lambda}^1)(x,\xi) \upalpha  \\
&\quad = 2\pi|\xi|^{-1} \int_{\omega \in S_x S^{n-1}_{>\beta'}} \alpha_k \omega^k \delta(|\xi|^{-1} \xi(\omega)) e^{-2\lambda \tau(x,-\omega)}d\omega\\
&\sigma ((T_\lambda^1)^* T_{\lambda}^1)(x,\xi) \upalpha \\
&\quad  =  2\pi|\xi|^{-1}\left( \int_{\omega \in S_x S^{n-1}_{>\beta'}}  \alpha_k \omega^k \delta(|\xi|^{-1} \xi(\omega))e^{-2\lambda \tau(x,-\omega)} d\omega \right)\delta_{j,k}dx^j.
\end{aligned}
\end{equation}

In the above equations $\tau(x,\omega)$ denotes the time it takes for a geodesic starting at $x$ with initial velocity $\omega$ to reach the boundary $\partial S_{>\beta'}^{n-1}$. Note that we can conclude from the above that $\{E_\lambda\}_{\lambda\in \R}$ is a family of operators depending continuously on the parameter $\lambda\in \R$ with respect to the topology given by the operator norm from $H^{-1}(S_{>\beta'}^{n-1}) \to L^2(S_{>\beta'}^{n-1})$.

Let $\chi\in C^\infty_0(S^{n-1}_{>0})$ be a cutoff function which is identically 1 on $S_{>(\beta + \beta')/2}^{n-1}$ but with support in $S_{>\beta'}^{n-1}$. Use the identities of \eqref{normal op}, \eqref{Elambda}, and \eqref{integrals} we have that for all $[f,\upalpha]\in T_xS^{n-1}_{>(\beta'+\beta)/2} \oplus \C$:
\begin{align*}
&\langle [f,\upalpha], \sigma (\chi E_\lambda \chi)(x,\xi) [f,\upalpha]\rangle  \\
&=2\pi|\xi|^{-1}\chi^2(x)\left(|{\rm proj}_{\xi}(\upalpha)|^2 + \int_{\omega \in S_x S^{n-1}_{>\beta'}} |f+\alpha_k \omega^k|^2 \delta( \xi(\frac{\omega}{|\xi|})) e^{-2\lambda \tau(x,-\omega)}d\omega \right)\\
& \geq 2\pi |\xi|^{-1} \chi^2(x) \left(|{\rm proj}_{\xi}(\upalpha)|^2  +e^{c\lambda} \int_{S_xS^{n-1}_{>\beta'}} |f+ \alpha_k\omega^k|^2 \delta( \xi(\frac{\omega}{|\xi|}))d\omega\right).
\end{align*}
By using a change of variable $\omega \mapsto -\omega$, we can write
\begin{align*}
& 2\int_{S_xS^{n-1}_{>\beta'}} |f+ \alpha_k\omega^k|^2 \delta( \xi(\frac{\omega}{|\xi|}))d\omega \\
&=  \int_{S_xS^{n-1}_{>\beta'}} |f+ \alpha_k\omega^k|^2 \delta( \xi(\frac{\omega}{|\xi|}))d\omega + \int_{S_xS^{n-1}_{>\beta'}} |f- \alpha_k\omega^k|^2 \delta( \xi(\frac{\omega}{|\xi|}))d\omega  \\
&  = 2  \int_{S_xS^{n-1}_{>\beta'}}f^2+ (\alpha_k\omega^k)^2 \delta( \xi(\frac{\omega}{|\xi|}))d\omega.
\end{align*}
This means that
\begin{align*}
& \langle [f,\upalpha], \sigma (\chi E_\lambda \chi)(x,\xi) [f,\upalpha]\rangle \\
&  \geq C_\lambda|\xi|^{-1} \chi^2(x) \left( {\rm proj}_{\xi}(\upalpha)|^2 + f^2+ \int_{S_xS^{n-1}_{>\beta'}} (\alpha_k\omega^k)^2 \delta( \xi(\frac{\omega}{|\xi|}))d\omega  \right)
\end{align*}
for some $C_\lambda$ which is uniformly bounded away from zero on a bounded sets of $\lambda\in \R$. Let us choose coordinates so that $\xi = |\xi| (1,0,\dots)$ and $\upalpha = (\alpha_1 ,\alpha')$ with $\alpha' = |\alpha'|(0,1,\dots)$. We then have that 
\begin{align*}
& \langle [f,\upalpha], \sigma (\chi E_\lambda \chi)(x,\xi) [f,\upalpha]\rangle \\
&  \geq C_\lambda |\xi|^{-1} \chi^2(x)\left( (\alpha_1)^2  +f^2 + |\alpha'|^2  \right) = C_\lambda |\xi|^{-1}\chi^2(x) (|\upalpha|_g^2 + f^2)
\end{align*}
with $C_\lambda$ uniformly bounded away from zero on bounded sets of $\lambda \in \R$. Therefore as long as $\lambda$ is restricted to a bounded set, $\chi E_\lambda \chi$ is a family of order $-1$ pseudodifferential operators which are uniformly elliptic on $S^{n-1}_{>(\beta'+\beta)/2}$.

Let $\tilde\chi\in C^\infty_0(S^{n-1}_{>0})$ such that $\tilde\chi =1$ on $S^{n-1}_{>\beta}$ with support contained in $S^{n-1}_{>(\beta'+\beta)/2}$. By ellipticity, there exists a family of order $1$ pseudodifferential operators $\tilde\chi P_\lambda \tilde\chi$ with principal symbol $\tilde\chi^2(x) \sigma(\chi E_\lambda\chi)^{-1}$ such that 
\begin{eqnarray}
\label{parametrix eq}
\tilde\chi P_\lambda \tilde \chi \chi E_\lambda \chi = \tilde\chi^2 + \tilde\chi K_\lambda\chi
\end{eqnarray}
with $K_\lambda : H^{l}(S^{n-1}_{>(\beta'+\beta)/2}) \to H^k(S^{n-1}_{>(\beta'+\beta)/2})$ for any $l,k\in \R$. This family of operators depend continuously on the parameter $\lambda \in \R$ in the operator norm topology. Therefore, for any $[f,\upalpha] \in H^{-1}(S^{n-1}_{>\beta})$, we have
$$\| [f,\upalpha] \|_{H^{-1}(S^{n-1}_{>\beta})} \leq C( \|E_\lambda [f,\upalpha]\|_{L^2(S^{n-1}_{>\beta'})} + \| \tilde \chi K_{\lambda} [f,\upalpha]\|_{H^{-1}(S^{n-1}_+}).$$
To apply Lemma \ref{abstract estimate}, we need to check that $E_\lambda : H^{-1}(S^{n-1}_{>\beta}) \to L^2(S^{n-1}_{>\beta'})$ is injective where $E_\lambda$ acts on elements of $H^{-1}(S^{n-1}_{>\beta})$ by first trivially extending them to become elements of $H^{-1}(S^{n-1}_{>\beta'})$. 

Suppose $[f,\upalpha] \in H^{-1}(S^{n-1}_{> \beta})$ is annihilated by $E_\lambda$. By \eqref{parametrix eq} we have that, when extended trivially outside of $S^{n-1}_{>\beta}$,
\[
\label{compact supp}
[f,\upalpha]\in C^{\infty}(S^{n-1}).
\]
Unpacking the definition of $E_\lambda$ we have that
$$ \| T_\lambda [f,\upalpha]\|_{L^2} + \int_{S^{n-1}} |\langle \epsilon D\rangle ^{-3/2} (\delta \upalpha)|^2= 0.$$
By Lemma \ref{bijection of jap bracket}, if $\epsilon >0$ is a fixed to be sufficiently small, 
\begin{eqnarray}
\label{solenoidal integral vanishes}
T_\lambda [f,\upalpha] = 0,\ \  \delta \upalpha = 0.
\end{eqnarray}
By \cite[Theorem 7.1]{DSFKSU},  we have that for $\lambda$ sufficiently small, there exists $\varphi\in C^\infty(S^{n-1}_{>\beta'})$ with Dirichlet boundary condition such that $f = \lambda \varphi$ and $\upalpha = d \varphi$. Combine this observation with \eqref{solenoidal integral vanishes} we have that $\Delta_g\varphi = 0$. Dirichlet boundary condition then forces $\varphi = 0$ and we have injectivity of $E_\lambda$. 

We now use Lemma \ref{abstract estimate} to deduce
$$\| [f,\upalpha] \|_{H^{-1}(S^{n-1}_{>\beta})} \leq C \|E_\lambda [f,\upalpha]\|_{L^2(S^{n-1}_{>\beta'})} $$
In the case when $\delta \upalpha  = 0$ we have estimate \eqref{shifted_version}.
\end{proof}

The normal operator argument is not needed in the absence of a magnetic potential as it was proved in \cite[Theorem 2.3]{CDSFR1}.
\begin{thm}\label{stability_estimates_dossantoscaroruiz129}
Let $0<\beta^\prime <\beta<1$ be as given in Lemma \ref{g_e_c_o_s}. Then there exist $\lambda_0>0$ and $C>0$ such that for all $0\leq \lambda \leq \lambda_0$ one has the following estimate
\[
\left\| f \right\|_{H^{-1/2}(S^{n-1}_{>\beta})} \leq C \left\| T_\lambda (f)\right\|_{L^2(\partial_+S(S^{n-1}_{>\beta^\prime}))}
\]
for all $f\in L^2(S^{n-1}_{>\beta^\prime})$ which are compactly supported on $S^{n-1}_{>\beta}$. 
\end{thm}

\section{Complex geometric optic solutions}\label{c_g_o_sulzlo}



In this section, we construct CGO solutions to $\mathcal{L}_{W, V}\,u=0$  in an open and bounded set with smooth boundary $B\subset \mathbb{R}^n$ so that $\overline{\Omega}\subset \subset B$, see Proposition \ref{CGO solutions_p}. The potentials $W$ and $V$ will play the role of $\chi_\Omega A$ and $\chi_\Omega q$, the extensions by zero out $\Omega$ of our original potentials. After some rotations and translations, without loss of generality, we can assume that 
\begin{equation}\label{geodesic_coordinates}
x_0=0, \quad 0\notin  ch(\overline{B}),\quad \overline{B}\subset \mathbb{R}^n_+:=\left\{ x\in \mathbb{R}\, :\, x_n>0\right\}.
\end{equation}
Note that these conditions remain true when replacing $B$ by $\Omega$. To construct CGO solutions and to be in line with the coordinates of the attenuated geodesic ray transform defined in Section \ref{preli}, we first introduce appropriate coordinates in $B$.

\subsection{Geodesic coordinates}
\label{g_ode_ic_ctes} For more details on this topic, see \cite[Section 2.1 and 3.1]{CDSFR1}. Recall that for $\hbar \in [0,1)$, we have defined
\[
S^{n-1}_{>\hbar}= \left\{ x\in S^{n-1}\, : \, x_n>\hbar\right\}\subset \mathbb{R}^n. 
\]
The manifold $S^{n-1}_{>\hbar}$ will always be considered with the canonical metric of $\mathbb{R}^n$. 
Thanks to \eqref{geodesic_coordinates}, we deduce that there exist $T>0$ and $\beta\in (0,1)$ (depending on $B$ and the distance from zero to $B$) such that every $x\in B$ can be written as 
\[
x= e^t w, \qquad \;  t=\log \left| x\right| \in (-T, T), \quad w= \dfrac{x}{\left| x\right|}\in S^{n-1}_{>\beta}.
\]
Let $0<\beta^\prime<\beta$. We claim that in turn $S^{n-1}_{>\beta}$ can be parametrized by geodesics on the unit sphere starting on $\partial S^{n-1}_{>\beta^\prime}$. Indeed, fixing an arbitrary $y\in \partial S^{n-1}_{>\beta^\prime}$, every $w\in S^{n-1}_{>\beta}$ may be written as
\[
w = (\cos \theta)y + (\sin \theta) \eta, \quad \theta \in  (\epsilon, \pi-\epsilon),\;  \eta \in y^{\perp}\cap S^{n-1}_{>0},
\]
for some $\epsilon\in (0, \pi)$ depending on $\beta$ and $\beta^\prime$ but independent of $y$. Here $\theta = d_{S^{n-1}}(y,w)=\arccos\langle y, w\rangle$ and $\eta= \theta^{-1}\exp_y^{-1}w$. Since $y\in \partial S^{n-1}_{>\beta^\prime}$, $d_{S^{n-1}} (y, \, \cdot \,)$ is smooth in $S^{n-1}_{>\beta}$. In those coordinates, the Lebesgue measure becomes
\[
dx = e^{nt} dt \, dw = e^{nt} (\sin\theta)^{n-2} \, dt  \,  d\theta \, d \eta,
\]
where $dw$ and $d\eta$ denote the Lebesgue measures on $S^{n-1}$ and $S^{n-2}$ provided with the canonical metrics of $\mathbb{R}^n$ and $\mathbb{R}^{n-1}$, respectively. Thus, $B$ can be imbedded into $\mathbb{R}\times S^{n-1}_{>\beta}$ or $\mathbb{R}\times \mathbb{R}\times S^{n-2}$. Furthermore, a straightforward computation shows that in the above coordinates, the Euclidean metric looks like
\[
|dx|^2 = e^{2nt}(dt^2 + g_{S^{n-1}})= e^{2nt}(dt^2 + d\theta^2 + (\sin\theta)^2\, g_{S^{n-2}}),
\]
where $g_{S^{n-j}}$ is the canonical metric on the hypersphere $S^{n-j}$, $j=1,2$.  
We summarize the above discussion.

\begin{lem}\label{g_e_c_o_s}
Let $B\subset \mathbb{R}^n$ be a bounded open set so that \eqref{geodesic_coordinates} is fulfilled. Let $0<\beta^\prime<\beta<1$ be as above. Fix $y\in \partial S^{n-1}_{>\beta^\prime}$. 
 Then there exist $T>0$ and $\epsilon\in (0, \pi)$ (both independent of $y$) such that 
\[
\begin{matrix}
\Psi_y:&B  &\rightarrow& (-T, T)\times (\epsilon, \pi-\epsilon)\times y^\perp \cap S^{n-1}_{>0} \\ 
 &x& \mapsto&(t, \theta, \eta)
\end{matrix}
\]
defines a change of variables, where
\[
\begin{aligned}
&x= e^t\, \gamma_{y, \eta}(\theta),\quad  \gamma_{y, \eta}(\theta) = (\cos \theta)y + (\sin \theta) \eta,\\
 &t=\log\left| x\right|, \quad \theta= d_{S^{n-1}}(y, x/\left| x\right|), \\
& \eta= (d_{S^{n-1}}(y, x/\left| x\right|))^{-1}\exp^{-1}_y(x/|x|).
\end{aligned}
\]
The function $d_{S^{n-1}} (y, \, \cdot \,)$ is smooth on $S^{n-1}_{>\beta}$, and the Euclidean metric in $\Psi_y(B)$ looks like 
\[
g= e^{2t}   \begin{pmatrix}
    \begin{array}{c|c}
 \begin{pmatrix}
1 &0 \\ 
 0& 1 \\
\end{pmatrix}_{2\times 2} & 0\, I_{(n-2)\times (n-2)}  \\
  \hline
  0\, I_{(n-2)\times (n-2)} & (\sin\theta)^2 \ g_{S^{n-2}}
    \end{array}
  \end{pmatrix}_{n\times n},
\]
where $I_{m\times m}$ denotes the $m\times m$ identity matrix. In particular, one has
\[
\left| dx\right|^2= e^{2t}(dt^2+  d\theta^2 + (\sin\theta)^2\, g_{S^{n-2}})
\]
and hence
\begin{equation}\label{det}
|g|:=\det g = e^{2 n t} (\sin \theta)^{2(n-2)} \det g_{S^{n-2}}.
\end{equation}
\end{lem}
\begin{rem}\label{laplace_beltrami_psi}
Fix $y\in \partial S^{n-1}_{>\beta^\prime}$. Let $\rho(x):=\log |x|+ i  d_{S^{n-1}}(y, x/\left| x\right|)$. Set $z:=t+i\theta$ and $2\partial_{\overline{z}}:= \partial_t + i\partial_\theta$. Writing $\widetilde{f}=f\circ \Psi_y^{-1}$, we get 
\[
\begin{aligned}
&\widetilde{\rho}=z,\quad  (\nabla \rho)^{\widetilde{}}= e^{-2t}(\partial_t+i \partial_\theta), \quad (\nabla \rho \cdot \nabla)^{\widetilde{}}= 2e^{-2t}\partial_{\overline{z}}, \\
& \widetilde{\Delta \rho}= 2e^{-2t}\, \partial_{\overline{z}}(\log(|g|^{1/2}e^{-2t})).
\end{aligned}
\]
Note that $\widetilde{\rho}$ only depends on $t$ and $\theta$, and it is independent of $\eta\in y^\perp \cap S^{n-1}$. It significantly reduces the computations behind the above identities. On the other hand, any magnetic potential $W\in W^{1, \infty}_c(B; \mathbb{C}^n)$ can be identified with $\sum_{j=1}^n W_j dx_j$. Thus, $W$ may be written in $\Psi_y$-coordinate as
\[
W= W_t \,dt + W_{\theta} \,d\theta+ W_{\eta} d\eta,
\]
where
\begin{equation}\label{potential_coordinates_geodecis_1}
\begin{aligned}
W_t(t, \theta, \eta)&= e^t\, W(e^t \, \gamma_{y, \eta}(\theta))\cdot \gamma_{y, \eta}(\theta),\\
W_\theta(t, \theta, \eta)&= e^t\, W(e^t \, \gamma_{y, \eta}(\theta))\cdot \dot{\gamma}_{y, \eta}(\theta), \; \dot{\gamma}= d\gamma / d\theta.
\end{aligned}
\end{equation}
If we denote local coordinates of $y^\perp \cap S^{n-1}$ by
\begin{align*}
\eta&= \Sigma (\alpha_1, \alpha_2, \ldots, \alpha_{n-2})\\
&= \left( \Sigma_1 (\alpha_1, \alpha_2, \ldots, \alpha_{n-2}), \Sigma_2(\alpha_1, \alpha_2, \ldots, \alpha_{n-2}), \ldots, \Sigma_n (\alpha_1, \alpha_2, \ldots, \alpha_{n-2})  \right),
\end{align*}
where $\Sigma_j$ are smooth maps for $j=1,2, \ldots, n$, and $\alpha_k \in \mathbb{R}$ with $k=1,2, \ldots, n-2$; then
\[
W_\eta d\eta = e^t\sin \theta \sum_{k=1}^n \sum_{j=1}^{n-2} W_k (e^t\gamma_{y, \eta}(\theta)) \partial_{\alpha_j} \Sigma_k\, d \alpha_j.
\]
For our purpose, we will only need $W_t$ and $W_\theta$. So it is enough to know that $W_\eta\in W^{1, \infty}_c(\Psi_y(B), \mathbb{C}^n)$. Finally, $W(e^t \ \cdot\ )$ is compactly supported in $S^{n-1}_{>\beta}$ for any fixed $t\in (-T, T)$.
\end{rem}

\subsection{CGO solutions in geodesic coordinates}

The construction of CGO solutions is based on standard techniques by combining the Hahn-Banach theorem with a Carleman estimate for a conjugate version of $\mathcal{L}_{W, V}$. To be in line with the usual Sobolev spaces involved in estimates of Carleman type, we introduce the semiclassical notation. For $\tau>0$, we define
\[
 \mathcal{L}_{W, V, \varphi}:=\tau^{-2}e^{\tau\varphi}\mathcal{L}_{W,V}\, e^{-\tau \varphi}, \quad \varphi(x)=\log |x|.
\]
Denote by $H^{1}_{scl}(B)$ the $H^1$-Sobolev space with semiclassical parameter $\tau^{-1}$, and its dual space by $H^{-1}_{scl}(B)$. Their norms are respectively defined by
\[
\left \| U \right \|_{H^{1}_{scl}(B)}= \left \| U \right \|_{L^2(B)}+ \left \| \tau^{-1}\nabla U \right \|_{L^2(B)},
\]
\[
\left \| U \right \|_{H^{-1}_{scl}(B)}=\underset{ \varpi  \, \in\, C^{\infty}_0(B)\setminus \left\{ 0 \right\}}{\sup} \dfrac{\left | \left \langle u, \varpi  \right \rangle_{L^2(B)} \right |}{\left \| \varpi  \right \|_{H^{1}_{scl}(B)}}.
\]

The following result is an immediate consequence of the Carleman estimate for Laplacian systems obtained in \cite[Lemma 2.1]{SaTz}; see also \cite[Proposition 2.3]{KU}.

\begin{prop} \label{carl_estimate}
Let  $W\in L^{\infty}_c(B; \mathbb{C}^n)$ and $V\in  L^{\infty}_c(B; \mathbb{C})$. There exists $\tau_0>0$ and $C>0$ (both depending on $n$, $B$, $\left\|  W \right\|_{L^\infty}$ and $\left\| V \right\|_{L^\infty}$) such that for all $F \in H^{-1}(B)$ there exists $R\in H^1(B)$ satisfying
 \[
 \mathcal{L}_{W, V, - \varphi}\, R =F \; \; \mbox{in} \; \; B,
\]
\[
\tau^{-1} \left \| R \right \|_{H^{1}_{scl}(B)} \leq C\left \| F \right \|_{H^{-1}_{scl}(B)},\quad \text{for all}\; \tau\geq \tau_0.
\]
\end{prop}
The following result was first proved in \cite{DSFKSjU} when $W\in C^2$ and $V\in L^\infty$. Later, it was improved in \cite{KnSa} by considering $W\in W^{1,n}\cap L^\infty$ and $V\in L^n$.
\begin{prop}\label{CGO solutions_p}
Let $W\in C^{1}_c(B; \mathbb{C}^n)$ and $V\in L^\infty_c(B; \mathbb{C})$. Assume $B\subset \mathbb{R}^n$ as in \eqref{geodesic_coordinates}. There exists $\tau_0>0$ and $C>0$ (both depending on $n$, $B$, $\left\|  W \right\|_{L^\infty}$ and $\left\| V \right\|_{L^\infty}$) such that the equation  $\mathcal{L}_{W , V}\, U=0$ in $B$ has a solution $U\in H^1(B)$ of the form
\[
U= e^{\tau(\varphi + i\psi)}\left( a + r  \right),\quad \tau \geq \tau_0,
\]
with the properties:
\begin{item}
\item[(i)] The functions $\varphi$ and $\psi$ are defined by 
\begin{equation}\label{eikonal_solution}
\varphi(x)=\log |x|, \qquad\quad \psi(x)= d_{S^{n-1}}(y, x/\left|x\right|),
\end{equation}
where $y\in S^{n-1}$ is chosen such that $\psi$ is smooth in $B$. 
\item[(ii)] The complex-valued function $a$ belongs to $W^{1,\infty}(B)$ and satisfies \footnote{ In equation below, the magnetic potential $W$ is view as a vectorial function so that $W\cdot D(\varphi + i\psi)$ stands for the usual inner product in $\mathbb{R}^n$.}
\[
\left[ 2D(\varphi + i\psi) \cdot D + 2W\cdot D(\varphi + i\psi) + D^2(\varphi + i\psi) \right]a=0,\; \mbox{in}\;\; B,
\]
\[
\left \| a  \right \|_{W^{\alpha, \infty}(B)} \lesssim \left \| W \right \|_{W^{\alpha,\infty}(B)}, \; \left | \alpha \right |\leq 1.
\]
\item[(iii)] The  function $r$ belongs to $H^1(B)$ and satisfies the estimate for all $\tau\geq \tau_0$
\[
 \left \|\partial^{\alpha} r  \right \|_{L^2(B)} \leq C \tau ^{ \left | \alpha \right |-1} \left \|  a\right \|_{H^1(B)},\; \left |\alpha  \right |	\leq 1.
\]
\end{item}
\end{prop}

\begin{proof} We are looking for solutions of the form 
\[
U=e^{\tau\rho }(a+r),\quad \rho= \varphi + i\psi,
\]
where $\psi$ is a smooth real-valued phase, $a$ is an amplitude, and $r$ is a correction term. A straightforward computation shows
\begin{equation}
\begin{aligned}
&e^{-\tau\rho} \tau^{-2} \mathcal{L}_{W,V} \, e^{\tau\rho}\\
&= D\rho \cdot D\rho  + \tau^{-1}\left(  2D\rho \cdot D + 2W\cdot D\rho + D^2\rho \right)  + \tau^{-2} \mathcal{L}_{W,V}.
\end{aligned}
\end{equation}
Motivated by this identity, $U$ satisfies $\mathcal{L}_{W, V}\, U=0$ if we solve in $B$ once at a time the equations for $\rho$, $a$ and $r$ in the following order:   
\begin{equation}\label{ro1}
D\rho \cdot D \rho =0,
\end{equation}
\begin{equation}\label{ro2}
\left(  2D\rho \cdot D + 2W\cdot D\rho + D^2\rho \right) a=0,
\end{equation}
\begin{equation}\label{ro3}
e^{-\tau\rho} \tau^{-2} \mathcal{L}_{W,V} \, e^{\tau\rho}  r = - \tau^{-2} \mathcal{L}_{W,V}\, a.
\end{equation}

\noindent \textit{Eikonal equation}. Note that \eqref{ro1} reads
\[
\left| \nabla \varphi \right| = \left| \nabla \psi\right|,\quad \nabla \varphi \cdot \nabla\psi =0, \quad \mbox{in}\; B.
\]
As shown in  \cite{KSU}, a solution $(\varphi, \psi)$ is given by \eqref{eikonal_solution}, where $y \in \partial S^{n-1}_{>\beta^\prime}$ is chosen as in Lemma \ref{g_e_c_o_s}, so $\psi$ is smooth in $B$.\\

\noindent  \textit{Transport equation}. We will solve \eqref{ro2} with the help of the $\Psi_y$-coordinates introduced in Section \ref{geodesic_coordinates}. Fix $y\in \partial S^{n-1}_{>\beta^\prime}$ and consider the $\Psi_y$-coordinates described in Lemma \ref{geodesic_coordinates}. Taking into account Remark \ref{laplace_beltrami_psi}, \eqref{ro2} becomes a $\partial_{\overline{z}}$-equation for $\widetilde{a}:=a\circ\Psi^{-1}_y$ and reads in $\Psi_y(B)$:
\begin{equation}\label{cx_0q}
\left[\partial_{\overline{z}}+ \partial_{\overline{z}}\left(\log(e^{-t} |g|^{1/4}) \right)+\frac{i}{2}(W_t+iW_\theta)\right]\widetilde{a} =0.
\end{equation}
Multiplying both sides by $e^{-t}|g|^{1/4}$, we deduce that
\begin{equation}\label{sfsfsdxb_f}
\widetilde{a}= |g|^{-1/4} e^t e^{\Phi}a_0=  |g_{S^{n-2}}|^{-1/4}\, e^{-(n-2)t/2}(\sin\theta)^{-\frac{n-2}{2}}\e^{\Phi}a_0
\end{equation}
 is a solution to \eqref{cx_0q} whenever $\partial_{\overline{z}} \,a_0=0$ and for any fixed $\eta\in y^\perp\cap S^{n-1}$
\begin{equation}\label{eq_delta_ug}
(\partial_{\overline{z}}\, \Phi)(\cdot, \eta) +\frac{i}{2}(W_t + iW_\theta)(\cdot, \eta)=0, \quad \mbox{in}\; \mathbb{R}^2.
\end{equation}
Since $W(\cdot, \eta)\in C^{1}_c(\mathbb{R}^2)$, Lemma \ref{delta_estimate} ensures that
\[
\Phi(\cdot, \eta):= -\frac{i}{2}(\mathcal{C}(W_t+iW_\theta))(\cdot, \eta)
\]
is a solution to \eqref{eq_delta_ug} satisfying for every $|\alpha|\leq 1$:
\begin{equation}\label{sfsfsdxb_f1}
\left\| \partial^\alpha\Phi \right\|_{L^\infty(\Psi_y(B))} \lesssim \left\| \partial^\alpha( W\circ \Psi_y) \right\|_{L^\infty(\Psi_y(B))} \lesssim \left\| \partial^\alpha W \right\|_{L^\infty(\Psi_y(B))},
\end{equation}
Here $\mathcal{C}$ denotes the Cauchy transform operator defined by \eqref{caushy_end}, identifying $\mathbb{R}^2$ with the complex plane $\mathbb{C}$. Hence, $a=\widetilde{a}\circ \Psi_y\in W^{1, \infty}(B)$ solves \eqref{ro2}. Finally, combining \eqref{sfsfsdxb_f} and \eqref{sfsfsdxb_f1}, we get
\begin{equation}\label{remain_term_q}
\left\| a\right\|_{H^1(B)}\lesssim \left\| \widetilde{a}\right\|_{H^1(\Psi_y(B))}\lesssim \left\| a_0\right\|_{H^1(\Psi_y(B))},
\end{equation}
where the implicit constant depends on $\left\|W\right\|_{W^{1, \infty}}$. \\

\noindent \textit{Correction term}. By the previous step,  $a=\widetilde{a}\circ \Psi_y\in W^{1, \infty}(B)$ and therefore $\mathcal{L}_{W,V}a\in H^{-1}(B)$. Hence Proposition \ref{carl_estimate} ensures the existence of $R\in H^1(B)$ satisfying
\[
\mathcal{L}_{W, V, -\varphi}\, R= -\tau^{-2} e^{i\tau\psi} \mathcal{L}_{W,V}\, a,
\]
\[
\tau^{-1} \left \| R \right \|_{H^{1}_{scl}(B)} \lesssim \left \| \tau^{-2} e^{i\tau\psi} \mathcal{L}_{W,V}\, a \right \|_{H^{-1}_{scl}(B)}.
\]
We now compute the $H^{-1}_{scl}$-norm. Let $\varpi \in C^{\infty}_0(B)$ be a non-vanishing function. Integration by parts with Cauchy-Schwarz's inequality yield
\begin{align*}
\left| \left \langle\tau^{-2} e^{i\tau\psi} \Delta\, a, \varpi   \right \rangle \right|& = \tau^{-2}\left| \left \langle \nabla a,   \nabla (e^{-i\tau\psi} \varpi ) \right \rangle\right|\\
& \lesssim \tau^{-1}\left| \left \langle \nabla a, e^{-i\tau\psi}  \nabla\psi\,  \varpi      \right \rangle   \right| +\tau^{-1} \left|  \left \langle   \nabla a, e^{-i\tau\psi}\tau^{-1} \nabla \varpi    \right \rangle   \right| \\
& \lesssim \tau^{-1} \left\| a\right\|_{H^1(B)}\left\| \varpi \right\|_{H^1_{scl}(B)}.
\end{align*} 
In the same fashion, we get bounds for $\nabla\, a$ and $a$ with $\tau^{-2}$ instead of $\tau^{-1}$ in the above last line. Combining these estimates, one gets
\begin{equation}\label{remain_termzq}
 \left \| \tau^{-2} e^{i\tau\psi} \mathcal{L}_{W,V}\, a \right \|_{H^{-1}_{scl}(B)} \lesssim \tau^{-1} \left\| a\right\|_{H^1(B)}.
\end{equation}
Consequently, $r=e^{-i\tau\psi}R\in H^1(B)$ is a solution to \eqref{ro3} satisfying the desired norm bounds. The proof is completed.
\end{proof}

\begin{rem}\label{form_a_j}
It will be convenient to consider the bounds of $a$ and $r$ in terms of $\Psi_y$-coordinates. To be more precise, throughout the proof of Proposition \ref{CGO solutions_p}, we have considered a family of solutions to the transport equation  \eqref{ro2} of the form $a:=\widetilde{a}\circ \Psi_y\in H^1(B)$, where $\widetilde{a}$ is given by \eqref{sfsfsdxb_f} with $a_0$ being any holomorphic function in the complex variable $z:=t+i\theta$. This freedom is crucial in our approach and will be used in the Appendix to remove the exponential term from the intermediate estimates. By combining \eqref{remain_term_q} and \eqref{remain_termzq}, we deduce the remainder term satisfies for $\tau\geq \tau_0$:
\[
 \left \|\partial^{\alpha} r  \right \|_{L^2(B)} \lesssim \tau ^{ \left | \alpha \right |-1} \left \|  a\right \|_{H^1(B)}    \lesssim  \tau ^{ \left | \alpha \right |-1} \left\| a_0\right\|_{H^1(\Psi_y(B))}\; \left |\alpha  \right |	\leq 1,
\]
where the implicit constant depends on $\left\|W\right\|_{W^{1, \infty}}$. 
\end{rem}

\section{Stability estimate for the magnetic potential}\label{sectio_four_st}
In this section, we derive a useful integral inequality relating the partial DN maps with the magnetic potentials, see Proposition \ref{magnetic_estimate_final} and Corollary \ref{cor:Phi_new_coordinates}. Finally, we prove Theorem \ref{SMP} employing a suitable attenuated geodesic ray transform associated to $A_1$ and $A_2$. From now on, for $j=1,2$, we denote $\Lambda_{A_j, q_j}$ and $\Lambda^\sharp_{A_j, q_j}$ by $\Lambda_j$ and $\Lambda^\sharp_j$, respectively. 

\subsection{Relating the partial DN maps with the magnetic potentials} \label{prioirbounsd} We start by stating the integral identity proved in \cite[Corollary 3.2]{Sun}.
\begin{lem} Assume that all conditions from the statement of Theorem \ref{SMP} hold. Suppose $u_j\in H^1(\Omega)$ satisfy $\mathcal{L}_{A_j, q_j}u_j=0$ in $\Omega$ with $j=1,2$. Then
\begin{equation}\label{al}
\begin{aligned}
& \left \langle  (\Lambda_1 - \Lambda_2)u_1,u_2 \right \rangle_{L^2(\partial\Omega)}\\
=&\int_{\Omega} \left[ (A_1-A_2)\cdot(Du_1 \overline{u}_2 + u_1 \overline{Du}_2) + (A_1^2-A_2^2+q_1-q_2)u_1\overline{u}_2\right]dx.
 \end{aligned}
\end{equation}
\end{lem}%
To exploit the information about $A_1-A_2$ encoded into this identity, we shall use the CGO solutions constructed in Section \ref{c_g_o_sulzlo}. Let $B\subset \mathbb{R}^n$ be an open and bounded set satisfying \eqref{geodesic_coordinates}. Now consider any compactly supported extension of $A_1$ denoted by $A_{1}^{ext}\in C^{1+\sigma}_c(B; \mathbb{C}^n)$. Since $A_1=A_2$ and $\partial_\nu A_1=\partial_\nu A_2$ on $\partial\Omega$, it follows that
\[
A_2^{ext}=\left\{\begin{matrix}
A_2 & \mbox{in }\; \Omega \\ 
A_1^{ext} & \mbox{in}\; \mathbb{R}^n\setminus \Omega
\end{matrix}\right.
\]
is a compactly supported extension of $A_2$ belonging to $C^{1+\sigma}_c(B; \mathbb{C}^n)$. Let $j=1,2$. By construction, these extensions satisfy 
\begin{equation}\label{cutt_off_func}
A_1^{ext}- A_2^{ext}= \chi_\Omega(A_1-A_2),
\end{equation}
\begin{equation}\label{bounded_extension}
\left\| A_j^{ext}\right\|_{W^{1, \infty}(B)}\lesssim \left\| A_j\right\|_{W^{1, \infty}(\Omega)}.
\end{equation}

By Proposition \ref{CGO solutions_p} applied to $W=A_j^{ext}$ and $V=\chi_{\Omega}q_j$, there exist $\tau_0>0$ and $U_j\in H^1(B)$ solutions to $\mathcal{L}_{A_j^{ext}, \,\chi_\Omega q_j}U_j=0$ in $B$ of the form
\begin{equation}\label{CgO_SoL:1}
\begin{aligned}
U_1&= e^{\tau(\varphi+i\psi)}(a_1+r_1),\\
U_2&= e^{\tau(-\varphi+i\psi)}(a_2+r_2),
\end{aligned}
\end{equation}
for all $\tau \geq \tau_0$. Recall that $\varphi$ and $\psi$ are defined by \eqref{eikonal_solution}. Moreover $a_j$ satisfy in $B$
\begin{equation}\label{eq:trans_original_coordinates}
\begin{aligned}
\left[ 2D(\varphi + i\psi) \cdot D + 2W\cdot D(\varphi + i\psi) + D^2(\varphi + i\psi) \right]a_1&=0,\\
\left[ 2D(-\varphi + i\psi) \cdot D + 2W\cdot D(-\varphi + i\psi) + D^2(-\varphi + i\psi) \right]a_2&= 0.
\end{aligned}
\end{equation}
By \eqref{bounded_extension}, we also have
\begin{equation}\label{first_uno}
\left \| a_j  \right \|_{W^{1, \infty}(B)} \lesssim \left \| A_j \right \|_{W^{1,\infty}(\Omega)}.
\end{equation}
Finally, $r_j$ belongs to $H^1(B)$ and 
\[
\left \|\partial^{\alpha} r_j  \right \|_{L^2(B)} \lesssim \tau ^{ \left | \alpha \right |-1} \left \|  a_j\right \|_{H^1(B)}\; , \left |\alpha\right |	\leq 1.
\]
Besides, by Remark \ref{form_a_j} and \eqref{sfsfsdxb_f}, the functions $a_j$ have the following form in $\Psi_y$-coordinates 
\begin{equation}\label{a_circ_1}
\begin{aligned}
\widetilde{a}_1& :=a_1\circ \Psi^{-1}_y=  |g_{S^{n-2}}|^{-1/4}\, e^{-(n-2)t/2}(\sin\theta)^{-\frac{n-2}{2}}\e^{\Phi_1}a_0,\\
\widetilde{a}_2&:=a_2\circ \Psi^{-1}_y=  |g_{S^{n-2}}|^{-1/4}\, e^{-(n-2)t/2}(\sin\theta)^{-\frac{n-2}{2}}\e^{\Phi_2},
\end{aligned}
\end{equation}
for any arbitrary and fixed $y\in \partial S^{n-1}_{>\beta^\prime}$ with $\partial_{\overline{z}} \,a_0=0$. Therefore
\begin{equation}\label{first_dos}
\left\| a_1\right\|_{H^1(B)}\lesssim \left\| a_0\right\|_{H^1(\Psi_y(B))}, \quad \left\| a_2\right\|_{H^1(B)}\lesssim 1.
\end{equation}
By \eqref{cutt_off_func}, we finally get
\begin{equation}\label{a_circ_2}
\partial_{\overline{z}}\, (\Phi_1+\overline{\Phi}_2) +\frac{i}{2}((\chi_\Omega(A_1-A_2))_t + i(\chi_\Omega(A_1-A_2))_\theta)=0, \; \mbox{in}\; \Psi_y(B).
\end{equation}

\begin{rem}\label{previous_coordinate} Although we are not yet proving the stability estimates for the electric potentials, we take advantage of the above computations to remark that sometimes it will be convenient to see \eqref{a_circ_2} in the original coordinates \eqref{eq:trans_original_coordinates}, from which we immediately deduce in $\Omega$:
\[
D(\varphi + i \psi)\cdot D \left( \left(\Phi_1 +\overline{\Phi}_2 \right) \circ \Psi_y+i\omega\right) + (A_1-A_2-\nabla \omega)\cdot D (\varphi + i \psi)=0,
\]
for every $\omega\in W^{1, \infty}(\Omega)$. In Section \ref{sta_electri_pot}, we will choose a particular $\omega$ being the function coming from the Hodge decomposition $A_1-A_2=F+ \nabla \omega$ for some vectorial function $F$, see Lemma \ref{hd}. The above equation can be solved using $\Psi_y$-coordinates and the Cauchy transform. Moreover, by \eqref{sfsfsdxb_f1}, we get
\[
\left\| \left(\Phi_1 +\overline{\Phi}_2 \right) \circ \Psi_y+i\omega \right\|_{L^\infty(\Omega)} \lesssim \left\| A_1-A_2-\nabla \omega \right\|_{L^\infty(\Omega)}.
\]
\end{rem}
\begin{prop}\label{magnetic_estimate_final} Taking into account the CGO solutions $U_j$ with $j=1, 2$, defined by \eqref{CgO_SoL:1}, we have
\begin{equation}\label{CgO_SoL:1_2}
\begin{aligned}
&\left| \int_{B} \chi_\Omega(A_1-A_2)\cdot D(\varphi+i\psi)\, a_1 \overline{a}_2 dx \right| \\
&\qquad \qquad \lesssim \left | \log \left \| \Lambda_1^\sharp- \Lambda_2^\sharp \right \| \right |^{-1} \left \| a_0\right \|_{H^1(\Psi_y(B))}.
\end{aligned}
\end{equation}
\end{prop}
The remaining part of this subsection will be devoted to proving this proposition. To do it, we need intermediate results. Note that $u_j:=U_j|_{\Omega}\in H^{1}(\Omega)$ are solutions in $\Omega$ to the original equations $\mathcal{L}_{A_j, q_j}u_j=0$, and so they satisfy the same bounds declared above. Inserting these solutions into \eqref{al}, we get
\begin{equation}\label{das01}
\begin{aligned}
& \tau^{-1} \left \langle (\Lambda_1- \Lambda_2)u_1, u_2  \right \rangle_{L^2(\partial\Omega)}\\
&= \tau^{-1}\int_{\Omega} \left[ (A_1-A_2)\cdot(Du_1 \overline{u}_2 + u_1 \overline{Du}_2)\right.\\
&\qquad \qquad \qquad  \qquad \qquad  \left. + (A_1^2-A_2^2+q_1-q_2)u_1\overline{u}_2\right] dx\\
&= \tau^{-1}\int_{B} \left[ \chi_\Omega(A_1-A_2)\cdot(DU_1 \overline{U}_2 + U_1 \overline{DU}_2)\right.\\
&\qquad \qquad \qquad  \qquad \qquad +\left. \chi_\Omega(A_{1}^2-A_{2}^2+q_1-q_2)U_1\overline{U}_2\right] dx\\
&= 2\int_{B} \chi_\Omega(A_1-A_2)\cdot D\rho\, a_1 \overline{a}_2 dx + \int_{B} \chi_\Omega(A_1-A_2)\cdot (M_1+M_2) dx\\
&\qquad +\tau^{-1} \int_{B}\chi_\Omega(A_1^2-A_2^2+q_1-q_2)U_1\overline{U}_2 dx, 
\end{aligned}
\end{equation}
where $\rho=\varphi+i\psi$ is given by \eqref{eikonal_solution}, and 
\begin{align*}
M_1=& D\rho\, r_1\overline{a}_2 + \tau^{-1}Da_1(\overline{a}_2+ \overline{r}_2) + \tau^{-1}Dr_1(\overline{a}_2+ \overline{r}_2 )+D\rho\,( a_1+r_1)\overline{r}_2,\\
M_2 =& D\rho\, (a_1+r_1) \overline{r}_2+\tau^{-1} a_1(\overline{Da}_2+\overline{Dr}_2)+ D\rho\, r_1\overline{a}_2  +\tau^{-1} r_1 (\overline{Da}_2 + \overline{Dr}_2).
\end{align*}
Combining \eqref{first_uno}-\eqref{first_dos}, it immediately follows that
\[
\left \| M_j  \right \|_{L^2(\Omega)} \lesssim \tau^{-1} \left \| a_0\right \|_{H^1(\Psi_y(B))}
\]
and by \eqref{das01}
\begin{equation}\label{das011}
\begin{aligned}
&2\left| \int_{B} \chi_\Omega(A_1-A_2)\cdot D(\varphi+i\psi)\, a_1 \overline{a}_2 dx \right|  \\
\lesssim& \, \tau^{-1} \left |    \left \langle (\Lambda_1- \Lambda_2)u_1, u_2  \right \rangle_{L^2(\partial\Omega)}      \right |  + \left \| M_1+M_2 \right \|_{L^2(B)} \\
& + \tau^{-1} \left \| e^{-\tau \varphi}U_1 \right \|_{L^2(B)}   \left \| e^{\tau \varphi}U_2 \right \|_{L^2(B)}\\
\lesssim &\,   \tau^{-1} \left |    \left \langle (\Lambda_1- \Lambda_2)u_1, u_2  \right \rangle_{L^2(\partial\Omega)}      \right | + \tau^{-1}\left \| a_0\right \|_{H^1(\Psi_y(B))}.
\end{aligned}
\end{equation}

Our task now is estimating the final boundary term on the right. Note that $\partial\Omega= \partial \Omega_{-,0}(x_0)\cup \partial \Omega_{+,0}(x_0)$. The terms coming from $\partial \Omega_{-,0}(x_0)$ are closely linked with $\Lambda^{\sharp}_1-\Lambda^{\sharp}_2$ but the ones coming from $\partial \Omega_{+,0}(x_0)$ do not, at least in a first inspection. To describe a suitable dependency, we need to make a more delicate analysis. It can be attained by using the below Carleman estimate with boundary terms derived in \cite{DSFKSjU} and \cite{KnSa}. 
\begin{prop}
\label{PCe}
Let $A\in W^{1,\infty}(\Omega; \mathbb{C}^n)$ and $q\in L^\infty(\Omega; \mathbb{C})$. There exists $\tau_0>0$ (depending on $n, \Omega,\left \| A \right \|_{W^{1,\infty}}, \left \| q \right \|_{L^\infty} $) such that  for all $u\in C^{\infty}(\Omega)\cap H^1_0(\Omega) $ the following estimate
\begin{equation}\label{Ce}
\begin{aligned}
&\left \| e^{-\tau\varphi} \partial_\nu u \right \|_{L^2_\omega(\partial \Omega_{+,0}(x_0))} +\tau^{1/2} \left \| e^{-\tau\varphi} u  \right \|_{L^2(\Omega)}   + \tau^{-1/2} \left \| e^{-\tau\varphi} \nabla u  \right \|_{L^2(\Omega)}\\
 &\quad   \lesssim    \tau^{-1/2} \left \| e^{-\tau\varphi} \mathcal{L}_{A,q} u \right \|_{L^2(\Omega)} + \left \| e^{-\tau\varphi} \partial_\nu u \right \|_{L^2_\omega(\partial \Omega_{-,0}(x_0))} 
\end{aligned}
\end{equation}
holds for all $\tau \geq \tau_0$. Here $\partial_\nu= \nu\cdot \nabla$ and $\varphi(x)=\log|x-x_0|$. The sets $\partial \Omega_{-,0}(x_0)$ and $\partial \Omega_{+,0}(x_0)$ are defined in \eqref{cset} and \eqref{cset1}, respectively. The norms of the boundary terms are weighted $L^2$-norms
\[
L^2_{\omega}(\Gamma )= L^2(\Gamma, \omega dS)\; , \; \omega(x)=\dfrac{ \left |  \left \langle \nu(x), x-x_0  \right \rangle \right |}{\left |  x-x_0 \right |^2}.
\]
\end{prop}
\begin{rem}\label{Cer} By a standard regularization method, \eqref{Ce} is still true for all $u$ in $H^1_0(\Omega)$ such that $\mathcal{L}_{A,q}u \in L^2(\Omega)$. The vanishing condition on $u$ is essential for deriving this result. Roughly speaking, this estimate tells us that it is possible to bound weighted terms coming from the shadowed face of the boundary $\partial \Omega_{+,0}(x_0)$ by the ones coming from the illuminated part $\partial \Omega_{-,0}(x_0)$. 
\end{rem}

\begin{lem} \label{reiamnir}
Let $\underline{c}> 1$ such that for all $x\in\partial\Omega$
\begin{equation}\label{c_underline}
\underline{c}^{-1} \leq |x-x_0|\leq \underline{c}.
\end{equation}
Then for $\tau>0$ large enough, we have
\begin{equation*}
\tau^{-1}\left |  \left \langle (\Lambda_1- \Lambda_2)u_1, u_2  \right \rangle_{L^2(\partial\Omega)} \right | 
\lesssim  \left(\underline{c}^{3\tau} \left \| \Lambda_1^\sharp - \Lambda_2^\sharp  \right \| +\tau^{-3/2}  \right) \left \| a_0\right \|_{H^1(\Psi_y(B))}.
\end{equation*}
\end{lem}
\begin{rem}
Condition \eqref{c_underline} is always satisfied since $x_0\notin ch(\overline{\Omega})$.
\end{rem}
\begin{proof} To remove in a safe way the weight $\omega$ from \eqref{Ce}, we first define the following subset of $\partial\Omega$ for $\delta>0$:
\[
F_{\delta}(x_0):=  \left \{ x\in \partial \Omega : \left \langle  \nu(x), x-x_0  \right \rangle \leq \delta \left | x-x_0 \right |^2 \right \}.
\]
Note that if $\delta$ is small enough, then $\partial \Omega_{-,0}(x_0)\subset\subset F_\delta(x_0)\subset \subset F$. Recall that $F$ is an open neighborhood of $\partial \Omega_{-,0}(x_0)$, see \eqref{cset}-\eqref{DNmapsz1}. Furthermore, in \eqref{DNmapsz1} we can assume that the boundary cutoff function $\chi$, in addition to be compactly supported in $F$, is equals $1$ on $F_{\delta}(x_0)$. Taking these facts into account and combining Cauchy-Schwarz's inequality with \eqref{partialnorm}, we obtain
\begin{equation}\label{partial_cutt_off}
\begin{aligned}
& \left| \left \langle (\Lambda_1- \Lambda_2)u_1, u_2  \right \rangle_{L^2(\partial\Omega)}\right| \\
& \leq\left \langle\chi (\Lambda_1- \Lambda_2)u_1, u_2  \right \rangle_{L^2(\partial\Omega)}  +  \left \langle(1-\chi) (\Lambda_1- \Lambda_2)u_1, u_2  \right \rangle_{L^2(\partial \Omega)}\\
 &  \lesssim \left \| \Lambda_1^\sharp- \Lambda_2^\sharp \right \| \left \| u_1 \right \|_{H^1(\Omega)}  \left \| u_2 \right \|_{H^1(\Omega)} \\
 & \quad +  \left \| e^{-\tau \varphi}  (\Lambda_1- \Lambda_2)u_1 \right \|_{L^2(\partial\Omega\setminus F_\delta(x_0))}  \left \| e^{\tau \varphi} u_2\right \|_{L^2(\partial\Omega\setminus F_\delta(x_0))}.
\end{aligned}
\end{equation}
We deal with the inner boundary product of the partial DN maps by using Proposition \ref{PCe}. Since $u_1$ does not necessarily vanish on the boundary, we introduce an auxiliary function $\widetilde{u}_1\in H^1(\Omega)$ satisfying
\begin{align*}
\label{artificial}
     \begin{cases}
            \mathcal{L}_{A_2,q_2}\widetilde{u}_1=0,  &  \mbox{in}\,\, \Omega
            \\ \widetilde{u}_1|_{\partial \Omega} = u_1|_{\partial \Omega}.
     \end{cases}
\end{align*}
Using the identity $\mathcal{L}_{A_2,q_2}(\widetilde{u}_1-u_1)= (\mathcal{L}_{A_1,q_1} - \mathcal{L}_{A_2,q_2}) u_1$, and since the term on the left belongs to $L^2(\Omega)$, we deduce $(\mathcal{L}_{A_1,q_1} - \mathcal{L}_{A_2,q_2}) u_1\in L^2(\Omega)$ as well. Thus, applying Proposition \ref{PCe} to $u:=u_1-\widetilde{u}_1\in H^1_0(\Omega)$, we  obtain
\begin{equation}\label{firststepw}
\begin{aligned}
&  \left \| e^{-\tau \varphi}  (\Lambda_1- \Lambda_2)u_1 \right \|_{L^2(\partial\Omega\setminus F_\delta(x_0))} = \left\|    e^{-\tau \varphi}  \partial_\nu(u_1-\widetilde{u}_1) \right\|_{L^2(\partial\Omega\setminus F_\delta(x_0))}  \\
   &\qquad \qquad \qquad \qquad \qquad   \leq \delta^{-1}\left \| e^{-\tau\varphi} \partial_\nu (u_1-\widetilde{u}_1) \right \|_{L^2_\omega(\partial \Omega_{+,0}(x_0))}  \\
   & \qquad \qquad \qquad \qquad \qquad   \lesssim    \tau^{-1/2}\delta^{-1} \left \| e^{-\tau\varphi}  (\mathcal{L}_{A_1,q_1} - \mathcal{L}_{A_2,q_2}) u_1 \right \|_{L^2(\Omega)} \\
   &\qquad \qquad \qquad \qquad \qquad   \quad +\delta^{-1}\left \| e^{-\tau\varphi} \partial_\nu (u_1-\widetilde{u}_1) \right \|_{L^2_\omega(\partial \Omega_{-,0}(x_0))}.
\end{aligned}
\end{equation}
Combining \eqref{cutt_off_func}-\eqref{first_dos} and the fact $\omega(x)\leq |x-x_0|^{-1}$ when $x\in \partial\Omega$, we get
\begin{align*}
 \left \| e^{-\tau\varphi}  (\mathcal{L}_{A_1,q_1} - \mathcal{L}_{A_2,q_2}) u_1 \right \|_{L^2(\Omega)} &\lesssim \left \| a_0\right \|_{H^1(\Psi_y(B))},\\
 \left \| e^{-\tau\varphi} \partial_\nu (u_1-\widetilde{u}_1) \right \|_{L^2_\omega(\partial \Omega_{-,0}(x_0))}& \lesssim 
\underline{c}^{\,\tau} \left \| \Lambda_1^\sharp - \Lambda_2^\sharp \right \| \left \| u_1 \right \|_{H^1(\Omega)},   \\
 \left \| e^{\tau \varphi} u_2\right \|_{L^2(\partial\Omega\setminus F_\delta(x_0))}& \lesssim 1,\\
   \left \| u_1 \right \|_{H^1(\Omega)} & \lesssim \tau\,\underline{c}^{\,\tau}\left \| a_0\right \|_{H^1(\Psi_y(B))}, \\
 \left \| u_2 \right \|_{H^1(\Omega)} & \lesssim \tau\,\underline{c}^{\,\tau}.
\end{align*}
We conclude the proof by combining these inequalities, replacing \eqref{firststepw} into \eqref{partial_cutt_off} and taking $\tau$ large enough.
\end{proof}
Previous estimates almost do the proof of Proposition \ref{magnetic_estimate_final}. Consider $\tau_0>0$ large enough such that for all $\tau\geq \tau_0$, both Lemma \ref{reiamnir} and $\underline{c}^{-\tau}\leq \tau^{-1}$ are satisfied. It is a simple matter to check that
\[
\tau:=   \dfrac{1}{4} |\log\underline{c}|^{-1} \left | \log \left \| \Lambda_1^\sharp- \Lambda_2^\sharp \right \| \right | \geq \tau_0 ,
\]
whenever
\[
 \left \|  \Lambda_1^{\sharp}-\Lambda_2^{\sharp} \right \| \leq e^{\underline{c}^{-4\tau_0}}.
\]
We end the proof Proposition \ref{magnetic_estimate_final} by using \eqref{das011} and Lemma \ref{reiamnir}.

\subsection{Stability for the magnetic potentials} The following result is an immediate consequence of Proposition \ref{magnetic_estimate_final} by making a $\Psi_y$-change of variables and considering Lemma \ref{g_e_c_o_s} together with \eqref{a_circ_1}-\eqref{a_circ_2}. 

\begin{cor} \label{cor:Phi_new_coordinates}
Let $0<\beta^\prime<\beta<1$ and $y\in \partial S^{n-1}_{>\beta^\prime}$ be as in Lemma \ref{g_e_c_o_s}. The equivalent version of \eqref{CgO_SoL:1_2} in $\Psi_y$-coordinates is given by
\begin{equation}\label{rremov_exp_tez}
\begin{aligned}
& \left| \int_{y^\perp\cap S^{n-1}} \int_{B\cap \Pi_{y, \eta}}  (A_t+iA_\theta)e^{i\Phi}a_0\, dz\,d\overline{z}\,d\eta\right| \\
&\qquad \lesssim  \left | \log \left \| \Lambda_1^\sharp- \Lambda_2^\sharp \right \| \right |^{-1} \left \| a_0\right \|_{H^1(\Psi_y(B))},
\end{aligned}
\end{equation}
where $z:=t+i\theta$, $\partial_{\overline{z}}\,a_0=0$, $A\equiv \chi_\Omega(A_1-A_2)$, $\Phi=\Phi_1+\overline{\Phi}_2$ and $\Pi_{y, \eta}$ denotes the plane generated by $y$ and $\eta\in y^{\perp}\cap S^{n-1}$. Finally, \eqref{a_circ_2} reads
\begin{equation}\label{phi_partia_uno}
\left( \partial_{\overline{z}}\,\Phi \right)(\cdot, \eta)+ \frac{i}{2}(A_t+iA_\theta)\left(\cdot, \eta\right)=0, \quad \mbox{in} \; B\cap \Pi_{y, \eta},
\end{equation}
for each $\eta\in y^\perp\cap S^{n-1}$. Here $A_t$ and $A_\theta$ are the first two components of $A$ in terms of $\Psi_y$-coordinates defined in Remark \ref{laplace_beltrami_psi}, see \eqref{potential_coordinates_geodecis_1}.
\end{cor}

The next step consists in proving that estimate \eqref{rremov_exp_tez} still holds without the exponential term on the left-hand side. For convenience to the reader, we postpone a detailed proof of this fact to the Appendix, see Theorem \ref{remov_esti_phi_1}. In particular, one can consider the holomorphic function $e^{i\lambda z}=e^{i\lambda(t+i\theta)}$ with $\lambda>0$ in  place of $a_0$ in \eqref{rremov_exp_tez}. Note that $A\equiv \chi_\Omega(A_1-A_2)$ belongs to $C^{1+\sigma}_c(B; \mathbb{C}^n)$, we deduce
\begin{align*}
 \int_{B\cap \Pi_{y, \eta}} e^{i\lambda (t+i\theta)} (A_t+iA_\theta)(t, \theta, \eta) dt d\theta&= \int_0^{\tau(y, \eta)} e^{-\lambda \theta} H(\gamma_{y, \eta}(\theta), \dot{\gamma}_{y, \eta}(\theta))d\theta\\
 &= (T_\lambda H_\lambda)(y, \eta),
\end{align*}
where $(y, \eta)\in \partial_+ S(S^{n-1}_{>\beta^\prime})$ and the function $H_\lambda:=f_\lambda + \alpha_\lambda$ is defined on $S(S^{n-1}_{>\beta})$ by $H_\lambda(w, \xi)=f_\lambda(w)+ {\alpha_{j, \lambda}}(w)\xi^j$. The functions $f_\lambda$ and $(\alpha_{j, \lambda})_{j=1}^n$ are compactly supported on $S^{n-1}_{>\beta}$ and defined by
\begin{align*}
f_\lambda(w)&= \int_{\mathbb{R}} e^{i\lambda t} e^t \, w\cdot A(e^t\, w) dt:= \int_{\mathbb{R}} e^{i\lambda t} e^t \, w\cdot \widetilde{A}(t, w) dt \\
\alpha_{j, \lambda} (w)&= i \int_{\mathbb{R}} e^{i\lambda t} e^t \, A_j(e^t w) dt:=  i \int_{\mathbb{R}} e^{i\lambda t} e^t \, \widetilde{A}_j(t,w) dt, \quad j=1, 2, \ldots, n,
\end{align*}
where $\widetilde{A}(t, w):=A(e^t w)$ with the convention adopted in formulae \eqref{potential_coordinates_geodecis_1}. Combining the fact that $\partial_+ S (S^{n-1}_{>\beta^\prime}))$ has finite measure (see Remark \ref{rema_finite_measure}) with Lemma \ref{continuity:attenuated_ray_transform_1} and Theorem \ref{remov_esti_phi_1}, we deduce that there exist $C:=C(\lambda)>0$ and $t_0\in (0,1)$ such that
\begin{align*}
|\langle (T_\lambda^*\, T_\lambda) H_\lambda, h \rangle_{L^2( S(S^{n-1}_{>\beta^\prime}))}|&=|  \langle T_\lambda H_\lambda, T_\lambda h \rangle_{L^2(\partial_+ S(S^{n-1}_{>\beta^\prime}))}  |\\
&\leq \left\| T_\lambda H_\lambda \right\|_{{L^2( \partial_+ S(S^{n-1}_{>\beta^\prime}))}}\left\|T_\lambda h \right\|_{{L^2( \partial_+ S(S^{n-1}_{>\beta^\prime}))}}\\
& \lesssim C(\lambda) \left | \log \left \| \Lambda_1^\sharp- \Lambda_2^\sharp \right \| \right |^{-t_0/10}\left\| h \right\|_{{L^2( S(S^{n-1}_{>\beta^\prime}))}}
\end{align*}
for all  $h\in C^{\infty}_c(S(S^{n-1}_{>\beta^\prime}))$, which immediately implies
\[
\left\|(T_\lambda^*\, T_\lambda) H_\lambda \right\|_{{L^2( S(S^{n-1}_{>\beta^\prime}))}} \lesssim C(\lambda) \left | \log \left \| \Lambda_1^\sharp- \Lambda_2^\sharp \right \| \right |^{-t_0/10}.
\]


Now we shall show how the solenoidal structure emerges when computing the left-hand side of this estimate. We will make a Hodge decomposition to  $\alpha_\lambda$ in $S^{n-1}_{>\beta}$. Since $\delta_{S^{n-1}}\alpha_\lambda \in H^{-1}(S^{n-1}_{>\beta})$, we deduce that there exists $\ss\in H^{1}_0(S^{n-1}_{>\beta})$ solving the equation
\begin{align*}
     \begin{cases}
           \Delta_{S^{n-1}} \ss= \delta_{S^{n-1}} \alpha_\lambda,  &  \mbox{in}\,\, S^{n-1}_{>\beta} \\
           \ss=0, & \mbox{on} \, \, \partial S^{n-1}_{>\beta}.
     \end{cases}
\end{align*}
The existence (and uniqueness) of such $ \ss$ is given by \cite[Lemma 1]{Ste}. We shall write $\alpha_\lambda= \alpha_\lambda^s + d \ss$, where $\alpha_\lambda^s=\alpha_\lambda- d \ss$ and $\delta_{S^{n-1}}\alpha_\lambda^s=0$. From the vanishing condition of $\ss$ on $\partial S^{n-1}_{>\beta}$ and the compact support property of $\alpha_{j, \lambda}$, $f_\lambda$ ($j=1,2, \ldots,n$), all of them can be extended by zero on $S^{n-1}_{>\beta^\prime}\setminus S^{n-1}_{>\beta}$. We still denote these extensions by the same letter. In particular $(f_\lambda + \lambda  \ss, \alpha_\lambda -d \ss)$ is a solenoidal pair, see Definition \ref{deft_solenoidal_partz}, and since $T_\lambda (\lambda  \ss + d \ss)=0$, we obtain
\[
T_{\lambda}^*\,T_{\lambda} (H_\lambda)= T_{\lambda}^*\,T_{\lambda}\left( (f_\lambda + \lambda  \ss) + (\alpha_\lambda - d \ss)\right).
\]
By Theorem \ref{stability_estimates_st_ra_tr}, there exists $\lambda_0$ such that for all $0\leq \lambda\leq \lambda_0$, we get
\[
\left\|f_\lambda + \lambda  \ss \right\|_{H^{-1}(S^{n-1}_{>\beta})} + \left\|\alpha_{\lambda} - d \ss \right\|_{H^{-1}(\Lambda^1S^{n-1}_{>\beta})} \lesssim C(\lambda) \left | \log \left \| \Lambda_1^\sharp- \Lambda_2^\sharp \right \| \right |^{-t_0/10}.
\]
A standard interpolation between the spaces $H^{-2}(\Lambda^1 S^{n-1}_{>\beta})$ and $L^{2}(\Lambda^1 S^{n-1}_{>\beta})$, and $H^{-2}(\Lambda^2 S^{n-1}_{>\beta})$ and $L^{2}(\Lambda^2 S^{n-1}_{>\beta})$, yields
\begin{equation}\label{dif_equatim}
\underset{|\lambda|\leq \lambda_0}{\sup} \left( \left\|df_\lambda + \lambda \alpha_\lambda \right\|_{H^{-1}(\Lambda^1 S^{n-1}_{>\beta})} + \left\| d\alpha_{\lambda} \right\|_{H^{-1}(\Lambda^2S^{n-1}_{>\beta})} \right) \lesssim \left | \log \left \| \Lambda_1^\sharp- \Lambda_2^\sharp \right \| \right |^{-t_0/5}.
\end{equation}

Let us compute the terms related to the $H^{-1}$-norms. 

\begin{align*}
df_\lambda + \lambda \alpha_\lambda& =\left( \int_{\mathbb{R}} e^{i\lambda t} \upsilon_k(t,w) dt  \right) dw_k= \left(  \widehat{\upsilon_k(\cdot, w)}(-\lambda)\right)dw_k,\\
 \upsilon_k(t,w) &= e^t\left( w_j \partial_{w_k} \widetilde{A}_j(t,w)- \partial_t \widetilde{A}_k(t,w)\right),  \quad k=1,2 \ldots, n. \\
d\alpha_\lambda &= \left(  \int_{\mathbb{R}} e^{i\lambda t} \varrho_{jk}(t,w)dt \right)dw_j dw_k= \left(  \widehat{ \varrho_{jk}(\cdot, w)}(-\lambda)\right)dw_j dw_k,\\
 \varrho_{jk}(t,w)& =  e^t\left( \partial_{w_k}\widetilde{A}_j(t, w) - \partial_{w_j}\widetilde{A}_k(t,w) \right),  \quad j, k=1,2 \ldots, n.
\end{align*}

The derivation under the integral in above identities is permitted by Fubini's theorem and since $A\in C^{1+\sigma}_c(B; \mathbb{C}^n)$. In addition, $\upsilon_k$ and $\varrho_{jk}$ belong to $L^1(\mathbb{R}; H^\sigma(S^{n-1}))$ for all $ j, k=1,2 \ldots, n$. An application of Riemann-Lebesgue theorem yields
\[
\left\|\widehat{ \upsilon_k(\cdot, \cdot)}(-\lambda)\right\|_{H^\sigma(S^{n-1})}+ \left\|\widehat{  \varrho_{jk}(\cdot, \cdot)}(-\lambda)\right\|_{H^\sigma(S^{n-1})} \lesssim 1,\quad  \lambda >0,
\] 
 where the implicit constant is uniform in $\lambda$. An interpolation between the spaces $H^{-1}(S^{n-1}_{>\beta})$ and $H^\sigma(S^{n-1}_{>\beta})$ combined with  \eqref{dif_equatim} yield us
\begin{align*}
\underset{|\lambda|\leq \lambda_0}{\sup} \left( \left\| \widehat{\upsilon_k(\cdot, \cdot)}(-\lambda) \right\|_{L^{2}(S^{n-1}_{>\beta})} +  \left\| \widehat{\varrho_{jk}(\cdot, \cdot)}(-\lambda) \right\|_{L^{2}(S^{n-1}_{>\beta})}\right)\\
\lesssim  \left | \log \left \| \Lambda_1^\sharp- \Lambda_2^\sharp \right \| \right |^{-\frac{t_0}{5(1+\sigma)}}.
\end{align*}
A direct application of Lemma \ref{lemma_estension_caro_ruiz_DSF} with $H=L^2(S^{n-1}_{>\beta})$ and first with $f= \upsilon_k$ and later with $f=\varrho_{jk}$ allows us to deduce 
\begin{equation}\label{last_hopefully}
\begin{aligned}
&\left\|\upsilon_k \right\|_{L^2(\mathbb{R}; L^2(S^{n-1}_{>\beta}))} + \left\|\varrho_{jk} \right\|_{L^2(\mathbb{R}; L^2(S^{n-1}_{>\beta}))} \\
&\qquad \qquad  \lesssim  \left |  \log \left | \log  \left \|\Lambda^{\sharp}_{1}- \Lambda^{\sharp}_{2}    \right \| \right |  \right | ^{-\frac{\sigma}{3(\sigma+1)}},
\end{aligned}
\end{equation}
for each $j,k=1, 2, \ldots,n$. The left-hand side of this estimate is related to $dA$ with $A=\chi_\Omega(A_1-A_2)$. Indeed, by using the chain rule combined with \eqref{det} with \eqref{c_underline}, we get 
\begin{align*}
\left\| dA_1-dA_2 \right\|_{L^2(\Omega)}^2&= \sum_{j,k=1}^n\int_{\mathbb{R}^n} | \partial_j A_k - \partial_k A_j |^2dx \\
&\leq \underline{c}^{n-2} \sum_{j,k=1}^n\int_{\mathbb{R}^n}|x-x_0|^{-n+2} | \partial_j A_k - \partial_k A_j |^2dx \\
& = \underline{c}^{n-2}  \sum_{j,k=1}^n \int_{\mathbb{R}} \int_{S^{n-1}_{>\beta}} |e^t (\partial_j A_k - \partial_k A_j)(x_0+ e^t w) |^2 dw\, dt \\
& \lesssim  \sum_{k=1}^n\left\|\upsilon_k \right\|_{L^2(\mathbb{R}; L^2(S^{n-1}_{>\beta}))}^2 +  \sum_{j,k=1}^n\left\|\varrho_{jk} \right\|_{L^2(\mathbb{R}; L^2(S^{n-1}_{>\beta}))}^2.
\end{align*} 
This estimate, combined with \eqref{last_hopefully}, ends the proof of Theorem \ref{SMP}.

\section{Stability estimate for the electric potential}\label{sta_electri_pot}

The goal of this section is to prove Theorem \ref{SEP}. We take advantage of the DN maps' invariance employing an appropriate Hodge decomposition to $A_1-A_2$. The following result has not been explicitly stated in \cite{Tz}. However, it can be easily deduced combining \cite[Lemma 6.2]{Tz} with the discussion just after its proof, see also \cite[estimate (23)]{Tz}.
\begin{lem}\label{hd}
Let $p>n$. There exist $\omega\in W^{3,p}(\Omega)\cap H_{0}^1(\Omega)$ and a positive constant $C$ such that
\[
\left \| A_1-A_2-\nabla \omega \right \|_{W^{1,p}(\Omega)} \leq C \left \|d(A_1-A_2)  \right \|_{L^p(\Omega)}
\]
and
\[
\left \| \omega \right \|_{W^{3,p}(\Omega)} \leq C \left \| A_1-A_2 \right \|_{W^{2,p}(\Omega)}.
\]
Here
\[
\left\| dA \right\|_{L^p(\Omega)}:= \sum_{1\leq j<k\leq n}\left\| \partial_j A_k - \partial_k A_j \right\|_{L^p(\Omega)}.
\]
\end{lem}
A direct application of Morrey's inequality yields
\begin{equation}\label{ggg1}
\left \| A_1-A_2-\nabla \omega \right \|_{C^{0, 1-\frac{n}{p}} (\overline{\Omega})} \leq C \left  \|d(A_1-A_2)  \right \|_{L^p(\Omega)}
\end{equation}
and
\begin{equation}\label{ggg2}
\left \|\partial^\alpha \omega \right \|_{L^{\infty}(\Omega)}  \leq C \left \| A_1-A_2 \right \|_{W^{2,p}(\Omega)}, \quad |\alpha|\leq 2.
\end{equation}
\begin{lem}\label{refzaqw2}
 Consider $\omega$ as in Lemma \ref{hd}. Let $j=1,2$. Define $\widetilde{A}_j:= A_j +(-1)^j\nabla \omega / 2$. Assume $u_j:=U_j|_{\Omega}\in H^{1}(\Omega)$ being the solution of $\mathcal{L}_{A_j, q_j}u_j=0$, where $U_j$ is given by \eqref{CgO_SoL:1}. Then $\widetilde{U}_j:= e^{(-1)^ji\omega/2}u_j\in H^1(\Omega)$ is a solution of $\mathcal{L}_{\widetilde{A}_j, q_j}\widetilde{U}_j=0$. Moreover, we have the DN map identity $\Lambda_{A_j,q_j}= \Lambda_{\widetilde{A}_j,q_j}$.
\end{lem}
\begin{proof} This result is an immediate consequence of the following identities
\[
e^{i\omega/2}\mathcal{L}_{A_1,q_1}e^{-i\omega/2}= \mathcal{L}_{\widetilde{A}_1, q_1} \; , \; e^{-i\omega/2}\mathcal{L}_{A_2,q_2}e^{i\omega/2}= \mathcal{L}_{\widetilde{A}_2, q_2},
\]
\[
e^{i\omega/2}\Lambda_{A_1,q_1}e^{-i\omega/2}= \Lambda_{\widetilde{A}_1, q_1} \; , \; e^{-i\omega/2}\Lambda_{A_2,q_2}e^{i\omega/2}= \Lambda_{\widetilde{A}_2, q_2}.
\]
which can be found, for instance, in \cite[Lemma 3.1]{KU}. Since $\omega|_{\partial\Omega}=0$, we easily deduce $\Lambda_{A_j,q_j}= \Lambda_{\widetilde{A}_j,q_j}$.
\end{proof}

By \eqref{CgO_SoL:1}, for a fixed $y\in \partial S^{n-1}_{>\beta^\prime}$ and for $\tau>0$ large enough, we have 
\[
\begin{aligned}
\widetilde{U}_1&= e^{-i\omega/2}e^{\tau(\varphi+i\psi)}(a_1+r_1),\\
\widetilde{U}_2&= e^{i\omega/2}e^{\tau(-\varphi+i\psi)}(a_2+r_2),
\end{aligned}
\]
with $\varphi$ and $\psi$ are defined by \eqref{eikonal_solution}, and $a_j$, $r_j$ satisfy \eqref{eq:trans_original_coordinates}-\eqref{first_dos}. Note that in $\Psi_y$-coordinates, $a_j$ has the form given by \eqref{a_circ_1} with $a_0$ being any holomorphic function in $z:=t+i\theta$, that is, $\partial_{\overline{z}}\,a_0=0$. Taking into account the notation from Lemma \ref{refzaqw2}, we deduce
\begin{equation}\label{algeea}
\begin{aligned}
& \left \langle  (\Lambda_1 - \Lambda_2)\widetilde{U}_1,\widetilde{U}_2 \right \rangle_{L^2(\partial\Omega)}=\int_{\Omega} \left[ (\widetilde{A}_1-\widetilde{A}_2)\cdot(D\widetilde{U}_1 \overline{\widetilde{U}}_2 + U_1 \overline{D\widetilde{U}}_2) \right. \\
&\qquad \qquad \qquad \qquad \qquad \qquad \qquad   \left. + (\widetilde{A}_1^2-\widetilde{A}_2^2+q_1-q_2)\widetilde{U}_1\overline{\widetilde{U}}_2\right]dx.
 \end{aligned}
\end{equation}
Our next task will be to get information of $q_1-q_2$ from this identity.  We expect to obtain a suitable attenuated geodesic ray transform of $q_1-q_2$, by following similar computations made in proving Theorem \ref{SMP}. In this way, one gets proper bounds in $L^2(\Omega)$ of $ (\widetilde{A}_1-\widetilde{A}_2)\cdot(D\widetilde{U}_1 \overline{\widetilde{U}}_2 + \widetilde{U}_1 \overline{D\widetilde{U}}_2)$ and $ (\widetilde{A}_1^2-\widetilde{A}_2^2)\widetilde{U}_1\overline{\widetilde{U}}_2$. Both terms have the common factor $\widetilde{A}_1-\widetilde{A}_2= A_1-A_2-\nabla \omega$, which can be related with our previous estimate for the magnetic part. Picking any $p>n$, we claim
\begin{equation}\label{hhh5}
\left \| A_1-A_2-\nabla \omega \right \|_{C^{0, 1-\frac{n}{p}}(\overline{\Omega})}  \lesssim \left | \log \left | \log  \left \|\Lambda^{\sharp}_{1}- \Lambda^{\sharp}_{2}    \right \| \right |   \right |^{-\kappa/(2p)}.
\end{equation}
Indeed, by interpolation ($2<p<2(p-1)$) we get
\[
\left \| d(A_1-A_2) \right \|_{L^p(\Omega)}\lesssim \left \| d(A_1-A_2) \right \|_{L^2(\Omega)}^{1/p} \left \| d(A_1-A_2) \right \|_{L^{2(p-1)}(\Omega)}^{1-1/p}.
\]
Consequently, the claimed estimate follows from Theorem \ref{SMP} and \eqref{ggg1}. Next, from identity \eqref{algeea} and taking into account Remark \ref{form_a_j} and the form of $a_j$ in $\Psi_y$-coordinates given by \eqref{a_circ_1}, we deduce
\begin{multline*}
\left| \int_{\Omega}(q_1-q_2) e^{ \left(\Phi_1 +\overline{\Phi}_2 \right) \circ \Psi_y+i\omega} a_0\circ \Psi_y \, dx  \right| \\
\leq \left(  \tau \underline{c}^{3\tau}\left \| \Lambda_1^{\sharp}-  \Lambda_2^{\sharp} \right \| + \tau^{-1/2} + \tau \left\| A_1-A_2-\nabla \omega \right\|_{L^\infty(\Omega)} \right) \left\| a_0\right\|_{H^1(\Psi_y(B))},
\end{multline*}
where we have used \eqref{sfsfsdxb_f1}, \eqref{first_uno}, \eqref{first_dos}, and Lemma \ref{reiamnir}. Here $\underline{c}>1$ is defined in \eqref{c_underline}. We now remove the exponential term $e^{\Phi_1+\overline{\Phi}_2+i\omega}$. This part is more manageable than the magnetic case, and it will follow by the identity
\begin{align*}
 \int_{\Omega}(q_1-q_2) a_0\circ \Psi_y\, dx =&   \int_{\Omega}( 1-e^{ \left(\Phi_1 +\overline{\Phi}_2 \right) \circ \Psi_y+i\omega} )(q_1-q_2)a_0\circ \Psi_y \, dx\\
&  +   \int_{\Omega} (q_1-q_2) e^{ \left(\Phi_1 +\overline{\Phi}_2 \right) \circ \Psi_y+i\omega}a_0\circ \Psi_y\, dx.
\end{align*}
By Remark \ref{previous_coordinate}, we deduce
\begin{multline*}
 \left | \int_{\Omega}( 1-e^{ \left(\Phi_1 +\overline{\Phi}_2 \right) \circ \Psi_y+i\omega} )(q_1-q_2) a_0\circ\Psi_y dx  \right | \\
 \lesssim   \left \| \left(\Phi_1 +\overline{\Phi}_2 \right) \circ \Psi_y+i\omega\right \|_{L^\infty(\Omega)}\left\|   (q_1-q_2) a_0\circ\Psi_y  \right\|_{L^1(\Omega)} \\
\lesssim  \left \| A_1-A_2-\nabla\omega \right \|_{L^\infty(\Omega)} \left \| a_0\right \|_{L^2(\Psi_y(B))},
\end{multline*}
where we have used the inequality
\[
\left | e^{a}-e^{b} \right |\leq \left | a-b \right | e^{\max \left \{ \Re\, a, \Re\, b \right \}} \; \; , \; \; a,b \in \mathbb{C}.
\]
Hence we obtain 
\begin{multline*}
\left| \int_{\Omega}(q_1-q_2) a_0\circ \Psi_y \, dx  \right| \\
\leq \left(  \tau \underline{c}^{3\tau}\left \| \Lambda_1^{\sharp}-  \Lambda_2^{\sharp} \right \| + \tau^{-1/2} + \tau \left\| A_1-A_2-\nabla \omega \right\|_{L^\infty(\Omega)} \right) \left\| a_0\right\|_{H^1(\Psi_y(B))}.
\end{multline*}
By \eqref{hhh5} and choosing $\tau\geq \tau_0>0$ as
\[
\tau:=   \dfrac{1}{8} |\log\underline{c}|^{-1}  \left |  \log \left | \log \left \| \Lambda_1^\sharp- \Lambda_2^\sharp \right \| \right |\right |^{\kappa/(3p)}\geq \tau_0
\]
with $\tau_0$ satisfying
\[
\left \| \Lambda_1^{\sharp}-  \Lambda_2^{\sharp} \right \| \leq  e^{-e^{ \left( 8\tau_0\log   \underline{c}\right)^{3p \kappa^{-1}}}},
\]
we get
\[
\left| \int_{\Omega}(q_1-q_2) a_0\circ \Psi_y \, dx  \right|\lesssim \left | \log \left | \log  \left \|\Lambda^{\sharp}_{1}- \Lambda^{\sharp}_{2}    \right \| \right |   \right |^{-\kappa/(3p)}  \left\| a_0\right\|_{H^1(\Psi_y(B))}.
\]

We can now proceed analogously to the magnetic case in obtaining suitable information of $q_1-q_2$ through an appropriate attenuated geodesic ray transform. We first set $Q:= \chi_{\Omega}(q_1-q_2)$ and choosing $a_0=e^{i\lambda z}b$ in the previous estimate with $b$ being any smooth function on $y^\perp \cap S^{n-1}$, we deduce
\begin{equation}\label{int_alosmt_fin}
\begin{aligned}
\underset{y\in \partial S^{n-1}_{>\beta^\prime}}{\sup}\left\| \int_{B\cap \Pi_{y, \eta}} e^{i\lambda z}Q\circ \Psi_y^{-1}\, dzd\overline{z} \right\|_{H^{-1}(y^\perp\cap S^{n-1})} \\
 \lesssim \left | \log \left | \log  \left \|\Lambda^{\sharp}_{1}- \Lambda^{\sharp}_{2}    \right \| \right |   \right |^{-\kappa/(3p)}.
\end{aligned}
\end{equation}
Since $Q$ is compactly supported and recalling that $z=t+i\theta$, we have
\begin{multline*}
\int_{B\cap \Pi_{y, \eta}} e^{i\lambda z}Q\circ \Psi_y^{-1}\, dzd\overline{z} \\
= \int_0^\pi e^{-\lambda\theta}\left[ \int_{-\infty}^\infty e^{i\lambda t} e^t Q(e^t ((\cos \theta) y + (\sin \theta) \eta)) dt\right]d\theta\\
:= \int_0^\pi e^{-\lambda\theta} \mathcal{Q}_\lambda (\gamma_{y, \eta}(\theta))d\theta= (T_\lambda \mathcal{Q}_\lambda)(y, \eta),\qquad  \qquad  
\end{multline*}
where $\mathcal{Q}_\lambda$ is defined on $S^{n-1}_{>\beta^\prime}$ by
\begin{equation}\label{new_not}
\mathcal{Q}_\lambda (w)=   \int_{-\infty}^\infty e^{i\lambda t} e^t \widetilde{Q}(t, w) dt= \widehat{\widetilde{Q}(\cdot, w)}(-\lambda), \quad  \widetilde{Q}(t, w) = Q(e^t w).
\end{equation}
 In light of Lemma \ref{g_e_c_o_s}, $\mathcal{Q}_\lambda$ is compactly supported on $S^{n-1}_{>\beta}$. Furthermore, by interpolation and \eqref{int_alosmt_fin}, we have
\begin{multline}\label{last_iloy}
\left\| T_\lambda \mathcal{Q}_\lambda  \right\|_{L^2(\partial_+ S(S^{n-1}_{>\beta^\prime}) )}= \left\| T_\lambda \mathcal{Q}_\lambda  \right\|_{L^2(\partial S^{n-1}_{>\beta^\prime}; L^{2}(y^\perp\cap S^{n-1}))}\\
\lesssim \left\| T_\lambda \mathcal{Q}_\lambda  \right\|_{L^2(\partial S^{n-1}_{>\beta^\prime}; H^{-1}(y^\perp\cap S^{n-1}))}^{\frac{\sigma}{\sigma+1}} \left\| T_\lambda \mathcal{Q}_\lambda  \right\|_{L^2(\partial S^{n-1}_{>\beta^\prime}; H^{\sigma}(y^\perp\cap S^{n-1}))}^{\frac{1}{\sigma+1}}\\
\lesssim \left\| T_\lambda \mathcal{Q}_\lambda  \right\|_{L^\infty(\partial S^{n-1}_{>\beta^\prime}; H^{-1}(y^\perp\cap S^{n-1}))}^{\frac{\sigma}{\sigma+1}} \left\| T_\lambda \mathcal{Q}_\lambda  \right\|_{L^2(\partial S^{n-1}_{>\beta^\prime}; H^{\sigma}(y^\perp\cap S^{n-1}))}^{\frac{1}{\sigma+1}}\\
 \lesssim \left | \log \left | \log  \left \|\Lambda^{\sharp}_{1}- \Lambda^{\sharp}_{2}    \right \| \right |   \right |^{-\frac{\kappa\sigma}{3p(\sigma+1)}} \left\| T_\lambda \mathcal{Q}_\lambda  \right\|_{L^2(\partial S^{n-1}_{>\beta^\prime}; H^{\sigma}(y^\perp\cap S^{n-1}))}^{\frac{1}{\sigma+1}}.
\end{multline}
It remains to bound the last term on the right. Since $Q\in H^\sigma(\mathbb{R}^n)$, it certainly belongs to $L^1(\mathbb{R}; H^\sigma(S^{n-1}))$, and by Riemann-Lebesgue theorem
\[
\underset{\lambda>0}{\sup}\left\| \mathcal{Q}_\lambda\right\|_{H^\sigma(S^{n-1})} < +\infty.
\]
By Lemma \ref{continuity:attenuated_ray_transform_1}, we have the bound
\[
 \left\| T_\lambda \mathcal{Q}_\lambda  \right\|_{L^2(\partial S^{n-1}_{>\beta^\prime}; H^{\sigma}(y^\perp\cap S^{n-1}))} \lesssim  \left\| \mathcal{Q}_\lambda\right\|_{H^\sigma(S^{n-1})}.
\]
Hence, \eqref{last_iloy} and Theorem \ref{stability_estimates_dossantoscaroruiz129} ensures that 
\[
\underset{|\lambda|\leq \lambda_0}{\sup} \left\| \mathcal{Q}_\lambda \right\|_{H^{-1/2}(S^{n-1}_{>\beta})} \lesssim \left | \log \left | \log  \left \|\Lambda^{\sharp}_{1}- \Lambda^{\sharp}_{2}    \right \| \right |   \right |^{-\frac{\kappa\sigma}{3p(\sigma+1)}},
\]
for $\lambda_0>0$ small enough. One again, an interpolation between the spaces $H^{-1/2}(S^{n-1}_{>\beta})$ and $H^\sigma(S^{n-1}_{>\beta})$ with  \eqref{new_not} yield 
\[
\underset{|\lambda|\leq \lambda_0}{\sup} \left\| \widehat{\widetilde{Q}(\cdot, \cdot)}(-\lambda) \right\|_{L^{2}(S^{n-1}_{>\beta})} \lesssim \left | \log \left | \log  \left \|\Lambda^{\sharp}_{1}- \Lambda^{\sharp}_{2}    \right \| \right |   \right |^{-\frac{2\kappa\sigma^2}{3p(\sigma+1)(2\sigma+1)}}.
\]
A direct application of Lemma \ref{lemma_estension_caro_ruiz_DSF} with $H=L^2(S^{n-1}_{>\beta^\prime})$ and $f=\widetilde{Q}$ shows
\[
\left\|\widetilde{Q} \right\|_{L^2(\mathbb{R}; L^2(S^{n-1}_{>\beta}))} \lesssim  \left | \log \left |  \log \left | \log  \left \|\Lambda^{\sharp}_{1}- \Lambda^{\sharp}_{2}    \right \| \right |  \right | \right |^{-\frac{\sigma}{3(\sigma+1)}}. 
\]
We end the proof by relating the left-hand side of this estimate with $q_1-q_2$. It can be easily seemed by combining \eqref{c_underline} with the following estimate 
\begin{align*}
||q_1-q_2||^2_{L^2(\Omega)}&=\int_{\Omega} |q_1(x) - q_2(x) |^2dx\\
&  \leq \underline{c}^{n-2} \int_{\Omega} |x-x_0|^{-n+2}|q_1(x) - q_2(x) |^2dx \\
& =\underline{c}^{n-2} \int_{\mathbb{R}} \int_{S^{n-1}_{>\beta}} |e^tQ(x_0+e^t w) |^2dw dt\\
& =\underline{c}^{n-2} \left\|\widetilde{Q} \right\|_{L^2(\mathbb{R}; L^2(S^{n-1}_{>\beta}))}^2.
\end{align*}




\begin{appendix}
 
 \section{Cauchy transform and its properties}
  The aim of this Appendix is proving Theorem \ref{remov_esti_phi_1}, showing that it is possible to remove the exponential term from the left side of estimate \eqref{rremov_exp_tez}.  
  \begin{thm} \label{remov_esti_phi_1} Let $0<\beta<\beta^\prime<1$ be as in Lemma \ref{g_e_c_o_s} and $\lambda>0$. Then there exist $C:=C(\lambda)>0$ and $t_0\in (0,1)$ such that
\begin{equation}\label{emov_sti_hi_2}
\underset{\underset{\eta\in y^\perp\cap S^{n-1}}{y\in \partial S^{n-1}_{>\beta^\prime}}}{\sup}\left| \int_{B\cap \Pi_{y, \eta}} e^{i\lambda z} (A_t+iA_\theta)(z, \eta)\, dz\,d\overline{z} \right| \lesssim  C(\lambda)  \left | \log \left \| \Lambda_1^\sharp- \Lambda_2^\sharp \right \| \right |^{-t_0/10}.
\end{equation}
\end{thm}

To prove this result, we use similar arguments from \cite{DSFKSjU}, where the authors proved an identifiability result (if $\Lambda_1^\sharp= \Lambda_2^\sharp$ then $dA_1=dA_2$ and $q_1=q_2$). They removed the exponential term when the left side of \eqref{rremov_exp_tez} is identically zero using holomorphic extensions of complex functions satisfying a $\overline{\partial}$-equation like \eqref{phi_partia_uno}. Their method works satisfactorily in a pointwise setting. This is the main reason why we have/need $L^\infty$-bounds in \eqref{emov_sti_hi_2}. Next, we derive a quantitative version of their holomorphic argument, see Lemma \ref{fourier_coeff_k}. Recall the geodesic coordinates given by Lemma \ref{quantification_complex_argument}. The first part of this Appendix is related to the regularity of  solutions of a $\overline{\partial}$-equation in $\mathbb{C}$. Later, we introduce some Fourier estimates on $S^1$. Finally, we prove Theorem \ref{remov_esti_phi_1}.\\

 Throughout this work, we have met several times with an equation in $\mathbb{C}$ of the form
 \begin{equation}\label{caushy}
 \partial_{\overline{z}} \, \mathcal{G} =G, \qquad  z=t+i\theta , \qquad  2 \partial_{\overline{z}}:= \partial_t + i\partial_\theta,
 \end{equation}
where $G$ is a given complex-valued function. Note that if $\mathcal{G}$ is a solution, then its holomorphic perturbations are also solutions. We are particularly interested in solutions given by the so-called the Cauchy transform of $G$:
  \begin{equation}\label{caushy_end}
(\mathcal{C}G)(z):= \int_{\mathbb{C}} \frac{G(\xi)}{z-\xi} \,d\xi d \overline{\xi},
 \end{equation}
Now we study how the regularity of $\mathcal{C}G$ depends on the smoothness of $G$. The following result is a collection of \cite[Lemma 2.1]{Sun}, \cite[Lemma 4.6]{Sa1}, \cite[Theorem 4.3.13]{AIM}, \cite[Chapter V/Lemma 1]{Ahl} and \cite[Proposition B.3.1]{LPMthesis}.
\begin{lem}\label{delta_estimate} 
Let $k\in \mathbb{N}$ and $\gamma\in (0,1)$. Let $G\in W^{k, \infty}(\mathbb{C})$ with $\supp G\subset B_R(0)$, $R>0$. Then $\mathcal{C}G \in W^{k, \infty}(\mathbb{C})$ solves (\ref{caushy}) and satisfies 
\[
\left \| \mathcal{C}G  \right \|_{W^{k, \infty}(\mathbb{C})} \lesssim \left \| G \right \|_{W^{k, \infty}(\mathbb{C})}.
\]
Moreover, if  $G\in C^{k+\gamma}(\mathbb{C})$ then the following estimate holds for all $\tilde p>2/\gamma$:
\[
\left \| \mathcal{C}G  \right \|_{C^{k+1+\gamma- 2/\tilde p}(\mathbb{C})}\lesssim\left \| G \right \|_{C^{k+\gamma}(\mathbb{C})}. 
\]
The implicit constants in both previous estimates only depends on $R$. 
\end{lem}

Our approach involves functions depending on the complex variable $z$ and $\eta\in y^\perp\cap S^{n-1}$ for a fixed $y\in \partial S^{n-1}_{>\beta^\prime}\equiv S^{n-2}$. This identification will be considered throughout this Appendix. In this sense, for a function $G: \mathbb{C}\times S^{n-2}\to \mathbb{C}$ (in our case $G\in C^{1+\sigma}(\mathbb{C}\times S^{n-2})$) we define its Cauchy transform with respect to the first variable as
\[
(\mathcal{C}{G})(z, \eta)= \int_{\mathbb{C}}\dfrac{G(\xi, \eta)}{z-\xi}d\xi d\overline{\xi}.
\]
For any $\eta\in S^{n-2}$ and according to \cite[Theorem 4.3.10]{AIM}, the Cauchy transform of $G(\cdot, \eta)$ satisfy in $L^2(\mathbb{C})$ 
\begin{equation}\label{casi_1}
(\partial_{\overline{z}}\, \mathcal{C}G)(\cdot, \eta)= G(\cdot, \eta), \qquad (\partial_{z}\, \mathcal{C}G)(\cdot, \eta)= (\mathcal{S}G)(\cdot, \eta),
\end{equation}
where 
\[
(\mathcal{S}G)(z, \eta):= -\underset{\epsilon\to 0}{\lim}\int_{|z-\xi|>\epsilon} \frac{G(\xi, \eta)}{(z-\xi)^2} d\xi d\overline{\xi}
\]
is the Beurling transform with respect to the first variable of $G$. By \cite[identity (4.21)]{AIM}, it acts as an isometry in $L^2(\mathbb{C})$
\begin{equation}\label{casi_2}
\left\|(\mathcal{S}G)(\cdot, \eta) \right\|_{L^2(\mathbb{C})}= \left\|G(\cdot, \eta) \right\|_{L^2(\mathbb{C})}.
\end{equation}
\begin{lem}\label{new_delta_bar_estimate}
Let $G\in C^{1}(\mathbb{C}\times S^{n-2})$. Assume that there exists $R>0$ such that $\supp G(\cdot, \eta)\subset B_R(0)$ for all $\eta\in S^{n-2}$. Then there exists a positive constant $C$ such that
\[
\left\| \mathcal{C} G\right\|_{W^{1, \infty}(S^{n-2}; H^1(B_R(0))} \leq C \left\| G \right\|_{W^{1, \infty}(\mathbb{C}\times S^{n-2})},
\]
where
\begin{multline*}
\left\| \mathcal{C} G \right\|_{W^{1, \infty}(S^{n-2}; H^1(B_R(0))}:= \underset{\eta\in S^{n-2}}{\sup}\left( \left\| \mathcal{C} G(\cdot, \eta)\right\|_{L^2(B_R(0))}   \right.\\
\left.+ \left\| (\partial_{z}\, \mathcal{C} G)(\cdot, \eta) \right\|_{L^2(B_R(0))}  + \sum_{j} \left\| (\partial_{\eta_j} \mathcal{C} G)(\cdot, \eta)\right\|_{L^2(B_R(0))} \right). 
\end{multline*}
\end{lem}
\begin{proof}
Fix an arbitrary $\eta\in S^{n-2}$. Combining \cite[Theorem 4.3.12]{AIM} and the compactly supported property of $G(\cdot, \eta)$, we get 
\begin{align*}
\left\|(\mathcal{C}G)(\cdot, \eta) \right\|_{L^2(B_R(0))} \leq \left\|(\mathcal{C}G)(\cdot, \eta) \right\|_{L^2(B_{2R}(0))}\\
\leq \left\| G(\cdot, \eta)\right\|_{L^2(B_R(0))} \lesssim \left\| G\right\|_{W^{1, \infty}(\mathbb{C}\times S^{n-2})}.
\end{align*}
By \eqref{casi_1}-\eqref{casi_2}, we immediately deduce
\begin{align*}
 \left\|  (\partial_{z}\, \mathcal{C}G)(\cdot, \eta) \right\|_{L^2(B_R(0))}\leq \left\|  (\partial_{z}\, \mathcal{C}G)(\cdot, \eta) \right\|_{L^2(\mathbb{C})} \\
 = \left\| G(\cdot, \eta)\right\|_{L^2(B_R(0))} \lesssim \left\| G\right\|_{L^{\infty}(\mathbb{C}\times S^{n-2})}.
\end{align*}
Since the Cauchy transform is only related to the first variable, it follows that 
\begin{multline*}
\left\|( \nabla_\eta\, \mathcal{C}G)(\cdot, \eta) \right\|_{L^2(B_R(0))}=\left\|( \mathcal{C}\nabla_\eta\,G)(\cdot, \eta) \right\|_{L^2(B_R(0))} \\
\leq \left\|( \mathcal{C} \nabla_\eta\, G)(\cdot, \eta) \right\|_{L^2(B_{2R}(0))}  \leq 6 \left\|(\nabla_\eta\,G)(\cdot, \eta) \right\|_{L^2(B_R(0))}
 \lesssim \left\| G\right\|_{W^{1, \infty}(\mathbb{C}\times S^{n-2})}.
\end{multline*}
We conclude the proof by combining these estimates. 
\end{proof}

\section{Fourier transform on $S^1$ and related estimates}

For a given function $\mathbb{F}:S^1\subset\mathbb{C}\to \mathbb{C}$ we define its Fourier coefficient at $k\in \mathbb{Z}$ by 
\begin{equation}\label{fourier_coeff_k}
\widehat{\mathbb{F}}(k)= \int_{S^1} e^{-ik z} \mathbb{F}(z)dz.
\end{equation}
\begin{lem}\label{quantification_complex_argument}
Let $\mathbb{F}\in H^{m}(S^1)$ be a complex-valued function with $m>1/2$. Assume that their Fourier coefficients satisfy 
\[
|\widehat{\mathbb{F}}(-k)| \leq C \varepsilon\,  k, \quad \; k\in \mathbb{Z}_+,
\]
where $C>0$ and $0<\varepsilon<1$. Then the following serie 
\[
\widetilde{\mathbb{F}}(z):= \sum_{k=-\infty}^{-1}    \widehat{\mathbb{F}}(k)\, e^{ i k z}
\]
converges in $H^\beta(S^1)$ for all $\beta\in (0, m-1/2)$. Furthermore, one has
\[
\left\|  \widetilde{\mathbb{F}} \right\|_{H^\beta(S^1)} \lesssim  \varepsilon^{\frac{m-1/2-\beta}{2(m+1)}}.
\]
The implicit constant depends on $C$, $m$, and $\beta$, and it is independent of $\varepsilon$.
Finally, we have the inequality 
\[
\left\| \mathbb{F} - \widetilde{\mathbb{F}} \right\|_{H^{\widetilde{\beta}}(S^1)} \lesssim \left\| \mathbb{F} \right\|_{H^{\widetilde{\beta}}(S^1)}, \qquad \widetilde{\beta}\in [0,m].
\]
\end{lem}

\begin{rem}\label{Dini_condition}
If $\mathbb{F}\in C^{m}(S^1)$ is H\"older continuous, that is $m\in (0,1)$, then one has the uniformly convergent representation
\[
\mathbb{F}(z)= \sum_{k=-\infty}^{\infty} \widehat{\mathbb{F}}(k)\, e^{ i k z}= \sum_{k=-\infty}^{-1} \widehat{\mathbb{F}}(k)\, e^{ i k z} + \sum_{k=0}^{\infty} \widehat{\mathbb{F}}(k)\, e^{ i k z} = \widetilde{\mathbb{F}}(z)  + \sum_{k=0}^{\infty} \widehat{\mathbb{F}}(k)\, e^{ i k z}.
\]
This is true because any H\"older continuous function satisfies the so-called Dini integral condition, one necessary condition to have the above representation. We deduce that $\mathbb{F}$ can be decomposed as the sum of a holomorphic and an anti-holomorphic part from the identity $\mathbb{F}= \mathbb{F}- \widetilde{\mathbb{F}} + \widetilde{\mathbb{F}}$. The term $\mathbb{F}- \widetilde{\mathbb{F}}$ can be extended holomorphically into the interior of $S^1$ and Lemma \ref{quantification_complex_argument} gives $H^\beta$-bounds for $ \widetilde{\mathbb{F}}$ and $H^{\widetilde{\beta}}$-bounds for $\mathbb{F}- \widetilde{\mathbb{F}}$.
\end{rem}
\begin{proof}
Let $\langle k \rangle=(1+k^2)^{1/2}$. For any fixed $\beta\in (0, m-1/2)$, we have 
\[
|\langle k \rangle^\beta \widehat{\mathbb{F}}(k)| \lesssim \langle k \rangle^{-m + \beta}, \quad k\in \mathbb{Z}\setminus \left\{ 0\right\}.
\]
By hypothesis, we deduce
\[
|\langle k \rangle^\beta\widehat{\mathbb{F}}(-k)| \leq C \varepsilon\,  {\langle k \rangle}^{1+\beta}, \quad \; k\in \mathbb{Z}_+.
\]
By Parseval identity and picking a natural number $N=N(\varepsilon)$ (which will be fixed later) depending on $\varepsilon$, we get
\begin{align*}
\left\| \widetilde{ \mathbb{F}}\right\|_{H^\beta(S^1)}^2 = \sum_{k=1}^{+\infty} |\langle k \rangle^\beta \widehat{\mathbb{F}}(-k)|^2=  \sum_{k=1}^{N(\varepsilon)-1} |\langle k \rangle^\beta \widehat{\mathbb{F}}(-k)|^2+  \sum_{k=N(\varepsilon)}^{\infty} | \langle k \rangle^\beta\widehat{\mathbb{F}}(-k)|^2\\
 \lesssim \varepsilon^2  \sum_{k=1}^{N(\varepsilon)-1} \langle k \rangle^{2(1+\beta)} +   \sum_{k=N(\varepsilon)}^{\infty} \langle k \rangle^{2(-m+\beta)}
\lesssim \varepsilon^2 N(\varepsilon)^{2(1+\beta)+1} + N(\varepsilon)^{2(-m+\beta)+1}.
\end{align*}
Note that since $\beta\in (0, m-1/2)$, it follows that $2(-m+\beta)+1<0$. This fact is crucial to ensure the convergence of the serie $\sum_{k=N(\varepsilon)}^{\infty}$ on the right.
%
%
%
%
%
We claim that for each $0<\widetilde{\theta}<2$, there exists $N(\varepsilon)$ (also depending on $\widetilde{\theta}$) such that
\[
 \varepsilon^2 N(\varepsilon)^{2(1+\beta)+1} \leq \varepsilon^{\widetilde{\theta}}, \quad N(\varepsilon)^{2(-m+\beta)+1} \leq\varepsilon^{\widetilde{\theta}}.
\]
Indeed, both conditions are satisfied if
\[
\varepsilon^{\frac{\widetilde{\theta}}{2(-m+\beta)+1}}  \leq N(\varepsilon) \leq \varepsilon^{\frac{\widetilde{\theta}-2}{2(1+\beta)+1}},
\]
which holds whenever
\[
\widetilde{\theta}\leq  -(2(-m+\beta)+1)(m+1)^{-1}. 
\]
This proves the claim. It remains to prove the $H^{\widetilde{\beta}}$- bound for $\mathbb{F} - \widetilde{\mathbb{F}}$.  It easily follows by applying once again Parseval identity. Let $\widetilde{\beta}\in [0,m]$. Then
\begin{multline*}
\left\| \mathbb{F} - \widetilde{\mathbb{F}}\right\|_{H^{\widetilde{\beta}}(S^1)}^2 = \sum_{k=0}^\infty  \langle k \rangle^{2\widetilde{\beta}}| \widehat{\mathbb{F}}(k) |^2 \leq \sum_{k=-\infty}^\infty  \langle k \rangle^{2\widetilde{\beta}}| \widehat{\mathbb{F}}(k) |^2 = \left\| \mathbb{F} \right\|_{H^{\widetilde{\beta}}(S^1)}^2.
\end{multline*}
\end{proof}

\begin{lem}\label{quantification_complex_argument_5}
Let $0<\varepsilon <1/2$. Then $|\log (1+z)|< 2\varepsilon$ for all $|z|<\varepsilon$.
\end{lem}
\begin{proof}
Let $z\in B_\varepsilon(0)$. By Weierstrass M-test, we have the following power series representations
\[
\log (1+z)=  \sum_{k=1}^{+\infty}(-1)^{k+1}\, \dfrac{z^k}{k}= z\sum_{k=1}^{+\infty}(-1)^{k+1}\, \dfrac{z^{k-1}}{k}
\]
with uniform convergence in $B_\varepsilon(0)$. Hence, the result quickly follows since 
\[
\left| \sum_{k=1}^{+\infty}(-1)^{k+1}\, \dfrac{z^{k-1}}{k} \right| \leq\sum_{k=0}^{+\infty} \varepsilon^k = \frac{1}{1-\varepsilon}< 2.
\]
\end{proof}

Since the invertibility of the attenuated ray transform is only valid for small attenuations, which implies the knowledge of the Fourier transform of a suitable function related to the magnetic field and electric potential, we use the following result \cite[Lemma 4.1]{CDSFR1} to extend the Fourier transform information to the whole real line $\mathbb{R}$.
\begin{lem}\label{lemma_estension_caro_ruiz_DSF}
Let $H$ be a Hilbert space and $\sigma\in (0,1]$, there exists a positive constant $C$, depending on $\sigma$, such that for all $K>0$, all $0<\lambda_0\leq 1$ and all functions $f\in L^1_{comp}(\mathbb{R}; H)\cap H^\sigma(\mathbb{R};H)$ with values in $H$ such that
\[
\left\|f \right\|_{L^1(\mathbb{R}; H)} + \left\|f \right\|_{H^\sigma(\mathbb{R}; H)} \leq K,
\]
we have
\[
\left\|f \right\|_{L^2(\mathbb{R}; H)} \leq C \max(1,K)^2 e^{2L\lambda_0} \lambda_0^{-1/2 - 2\sigma} \left|  \log \underset{|\lambda|\leq \lambda_0}{\sup} \left\| \hat{f}(\lambda)\right\|_H \right|^{-\frac{\sigma}{3(\sigma+1)}},
\]
where $\supp \, f \subset [-L, L]$.
\end{lem}

\section{Removing the exponential term}
This part of the Appendix is devoted to proving Theorem \ref{remov_esti_phi_1}. Recall that $\Pi_{y, \eta}$ stands for the plane generated by $(y, \eta)\in \partial_+ S(S^{n-1}_{>\beta^\prime})$, where $0<\beta^\prime<1$ is given by Lemma \ref{g_e_c_o_s}. When trying to remove the exponential term from \eqref{rremov_exp_tez} using geodesic coordinates, we have to control the family of quantities associated with the Lebesgue measure of the sections $\Omega\cap \Pi_{y, \eta}$ with $(y, \eta)\in \partial_+ S(S^{n-1}_{>\beta^\prime})$. However, since we do not a priori the shape of $\Omega$, some elements of the family $(\Omega\cap \Pi_{y, \eta})_{(y, \eta)\in \partial_+ S(S^{n-1}_{>\beta^\prime})}$ could degenerate in the sense that their Lebesgue measure can be zero. For this reason, to overcome this technical difficulty, we shall consider a larger open subset $B$ containing $\Omega$ and so that it still satisfies the geometric condition \eqref{geodesic_coordinates}. The set $B$ has to be chosen so that the family $(B\cap \Pi_{y, \eta})_{(y, \eta)\in \partial_+ S(S^{n-1}_{>\beta^\prime})}$ behaves well in the sense that it does not contain any degenerate element. Roughly speaking, it can be done by approximating the plane $x_n=\beta^\prime$ with balls in $\mathbb{R}^n$ of radius sufficiently large and lying in $\left\{ x_n>\beta^\prime/2\right\}$. In fact, one can prove that there exist $y_0\in \mathbb{R}^n$ and $R_0>0$ such that $B_{R_0}(y_0)$ (the ball of radius $R_0$ and centre $y_0$) satisfies
\begin{equation}\label{rad_bound_0}
\Omega\subset \subset B_{R_0}(y_0) \subset \left\{ x_n>\beta^\prime/2\right\}.
\end{equation}
From now on, we assume $B$ to be $B_{R_0}(y_0)$. Moreover, one can also prove that there exists $S\gg 1$ such that 
\begin{equation}\label{rad_bound}
S^{-1}\leq R_{y, \eta}:=\sqrt{\left \langle y, y_0 \right \rangle^2 + \left \langle \eta, y_0 \right \rangle^2+ R_0^2 - |y_0|^2}\leq S,
\end{equation}
for all $(y, \eta)\in \partial_+ S(S^{n-1}_{>\beta^\prime})$. Somehow $R_{y, \eta}$ is proportional to the diameter of the set $B\cap\Pi_{y, \eta}$, see \eqref{change_variables_newzxcq}. Thus, condition \eqref{rad_bound} implies the no existence of degenerate elements in the family $(B\cap \Pi_{y, \eta})_{(y, \eta)\in \partial_+ S(S^{n-1}_{>\beta^\prime})}$.
\begin{rem}
A straightforward computation shows that $y_0=(R_0+\beta^\prime)e_n$, $R_0=2{\beta^\prime}^{-1}$ and
\[
S=\max \left\{{\beta^\prime}^{-2} ({\beta^\prime}^{2}+3)^{-1}, 8{\beta^\prime}^{-2} +4+{\beta^\prime}^{2}\right\}.
\]
satisfy conditions \eqref{rad_bound_0}-\eqref{rad_bound}. To prove \eqref{rad_bound}, one can use the fact that if $(y, \eta)\in \partial_+ S(S^{n-1}_{>\beta^\prime})$ then, by definition, $y\cdot e_n>\beta^{\prime}$, $y\cdot \eta=0$ and $|\eta|=1$.
\end{rem}

Let $0<\beta<\beta^\prime<1$ be as in Lemma \ref{g_e_c_o_s}. Then there exist $T>0$ and $\epsilon\in (0, \pi)$ --- both independent of $(y, \eta)\in \partial_+ S(S^{n-1}_{>\beta^\prime})$ --- such that 
\begin{equation}\label{change_variables_newzxcq}
\begin{aligned}
B \cap \Pi_{y, \eta}&   = \left\{(t,\theta) \in (-T, T)\times (\epsilon, \pi-\epsilon): \right.\\
&\qquad \qquad \qquad \qquad  \left. |e^{t+i\theta}- (\left \langle y, y_0 \right \rangle+i \left \langle \eta, y_0 \right \rangle)|<  {R_{y, \eta}}\right\},\\
\partial(B \cap \Pi_{y, \eta})&= \left\{(t,\theta) \in (-T, T)\times (\epsilon, \pi-\epsilon): \right.\\
& \qquad \qquad \qquad \qquad \left. |e^{t+i\theta}- (\left \langle y, y_0 \right \rangle+i \left \langle \eta, y_0 \right \rangle)|=  {R_{y, \eta}}\right\}.
\end{aligned}
\end{equation} 
\begin{rem}\label{rem_identi_1_i}
We make the natural identification between the variables $(t, \theta)\in \mathbb{R}^2$ and $z=t+i\theta \in \mathbb{C}$. In this sense, we shall also see $B \cap \Pi_{y, \eta}$ and its boundary $\partial(B \cap \Pi_{y, \eta})$ as subsets of the complex plane with references $y$ and $\eta$, which in turn can be seen as $1$ and the imaginary unit $i$ in the complex plane.
\end{rem}

Consider $a_0(z, \eta)=\widetilde{a}_0(z, \eta)b(\eta)$ in \eqref{rremov_exp_tez} with $\partial_{\overline{z}}\,\widetilde{a}_0(z, \eta)=0$ and $b$ being any smooth function on $y^\perp\cap S^{n-1}$.

\begin{lem} \label{iner_until_infuzz}
Let $y\in \partial S^{n-1}_{>\beta^\prime}$ and $p> 1$. There exists a universal constant $C>0$ such that
\[
\begin{aligned}\label{iner_until_infuz}
&\left\|  \int_{B\cap \Pi_{y, \cdot}}  (A_t+iA_\theta)(z, \cdot)\,e^{i\Phi(z, \cdot)}\, \widetilde{a}_0(z, \cdot)\, dz\,d\overline{z} \right\|_{W^{-1, \frac{2p}{p+1}}(y^\perp\cap S^{n-1})}\\
&\qquad \leq C  \left | \log \left \| \Lambda_1^\sharp- \Lambda_2^\sharp \right \| \right |^{-1}\left\| \widetilde{a}_0\right\|_{W^{1, 2p}\left(y^\perp\cap S^{n-1}; \, H^1(B\cap\Pi_{y, \eta}) \right)},
\end{aligned}
\]
where
\begin{equation}\label{uniform_bounded_R}
\begin{aligned}
&\left\| \widetilde{a}_0\right\|^{2p}_{W^{1, 2p}\left(y^\perp\cap S^{n-1}; \, H^1(B\cap\Pi_{y, \eta}) \right)}\\
&:=  \int_{y^\perp\cap S^{n-1}} \left( \left\| \widetilde{a}_0 (\cdot, \eta )\right\|^{2p}_{L^2(B\cap \Pi_{y, \eta})} +  \left\| \nabla_{z, \eta}\, \widetilde{a}_0 (\cdot, \eta )\right\|^{2p}_{L^2(B\cap \Pi_{y, \eta})} \right) d\eta.
\end{aligned}
\end{equation}
\end{lem}
\begin{proof}
H\"older's inequality applied to $p$ and $p/(p-1)$ implies
\[
\left\| a_0\right\|_{H^1(\Psi_y(B))} \lesssim \left\| \widetilde{a}_0\right\|_{W^{1, 2p}\left(y^\perp\cap S^{n-1}; \, H^1(B\cap\Pi_{y, \eta}) \right)}   \left\| b\right\|_{W^{1, \frac{2p}{p-1}}(y^\perp\cap S^{n-1})}.
\]
Indeed, since $1/p + (p-1)/p=1$, we have ($z:=t+i\theta$)
\begin{equation}\label{lp_s_n}
\begin{aligned}
\left\|a_0 \right\|^{2}_{L^2(\Psi_y(B))}=\int_{y^\perp \cap S^{n-1}} \left( \int_{B\cap \Pi_{y, \eta}} | \widetilde{a_0}(t, \theta, \eta)|^2 dtd\theta\right)|b(\eta)|^2 d\eta\\
\leq \left\| \int_{B\cap \Pi_{y, \cdot}} | \widetilde{a_0}(t, \theta, \cdot)|^2 dtd\theta\right\|_{L^p(y^\perp\cap S^{n-1})} \left\| b^2\right\|_{L^{\frac{p}{p-1}}(y^\perp\cap S^{n-1})},
\end{aligned}
\end{equation}
We now compute these norms.
\begin{align*}
\left\|  \int_{B\cap \Pi_{y, \cdot}} | \widetilde{a_0}(t, \theta, \cdot)|^2 dtd\theta\right\|_{L^p(y^\perp\cap S^{n-1})}^p& = \int_{y^\perp \cap S^{n-1}}\left\|\widetilde{a_0} (\cdot, \eta) \right\|^{2p}_{L^2(B\cap \Pi_{y, \eta})}d\eta\\
&=\left\| \widetilde{a}_0 \right\|_{L^{2p}(y^\perp \cap S^{n-1};\, L^2(B\cap\Pi_{y, \eta}))}^{2p}
\end{align*}
and
\[
\left\| b^2\right\|_{L^{p/(p-1)}(y^\perp\cap S^{n-1})}^{p/(p-1)} =  \left\| b\right\|_{L^{2p/(p-1)}(y^\perp\cap S^{n-1})}^{2p/(p-1)}.
\]
Combining these identities into \eqref{lp_s_n}, we get
\[
\left\|a_0 \right\|_{L^2(\Psi_y(B))}\leq \left\| \widetilde{a}_0 \right\|_{L^{2p}(y^\perp \cap S^{n-1};\, L^2(B\cap\Pi_{y, \eta}))}\left\| b\right\|_{L^{2p/(p-1)}(y^\perp\cap S^{n-1})}.
\]
We can continue in this fashion to get similar bounds for $\nabla_{z}a_0$ and $\nabla_{\eta}a_0$. The lemma follows by  inserting these estimates in \eqref{rremov_exp_tez} and by duality. 
\end{proof}

\begin{lem} \label{lem_interp_L_infty}  
Let $y\in \partial S^{n-1}_{>\beta^\prime}$. Let $\gamma\in (0,1/2)$ and $s>n$. Assume that $\widetilde{a}_0\in W^{1, s}(\Psi_y(B))$. Then 
\begin{align*}
& \left\|  \int_{B\cap \Pi_{y, \cdot}}  (A_t+iA_\theta)(z, \cdot)\,e^{i\Phi(z, \cdot)}\, \widetilde{a}_0(z, \cdot)\, dz\,d\overline{z} \right\|_{W^{\gamma, s}(y^\perp\cap S^{n-1})} \\
&\qquad \qquad \qquad\lesssim \left\| \widetilde{a}_0\right\|_{W^{1, s}\left(y^\perp\cap S^{n-1}; \, L^2(B\cap\Pi_{y, \eta}) \right)}.\\
\end{align*}
\end{lem}
\begin{proof}
We start by denoting
\[
\mathbb{A}(z, \eta):=(A_t+iA_\theta)(z, \eta)e^{i\Phi(z, \eta)}, \quad \chi_{\eta}(z):=\chi_{B\cap \Pi_{y, \eta}}(z) 
\]
and
\[
F(\eta)= \int_{\mathbb{C}} \chi_{\eta}(z)\, \mathbb{A}(z, \eta)\, \widetilde{a}_0(z, \eta)\, dz\,d\overline{z}.
\]
Since $y^\perp\cap S^{n-1}$ can be identified with $S^{n-2}$, we have
\begin{equation}\label{F_C_gamma}
\begin{aligned}
&\left\| F \right\|^s_{W^{\gamma, s}(y^\perp\cap S^{n-1})}= \int_{y^\perp\cap S^{n-1}}\left| F(\eta)\right|^sd\eta\\
&\qquad \quad \quad+  \int_{y^\perp\cap S^{n-1}} \int_{y^\perp\cap S^{n-1}}\dfrac{|F(\eta_1)- F(\eta_2)|^s}{|\eta_1-\eta_2|^{n-2+\gamma s}}d\eta_1d\eta_2 := I_1+I_2.
\end{aligned}
\end{equation}
Cauchy-Schwarz's inequality yields
\begin{align*}
I_1&\leq \int_{y^\perp\cap S^{n-1}}\left( \int_{B\cap \Pi_{y, \eta}} | \mathbb{A}(z, \eta)\,\widetilde{a}_0(z, \eta)|dzd\overline{z}\right)^sd\eta\\
& \leq \int_{y^\perp\cap S^{n-1}}\left\| \mathbb{A}(\cdot, \eta) \right\|^s_{L^2(B\cap \Pi_{y, \eta})} \left\|\widetilde{a}_0(\cdot, \eta) \right\|^s_{L^{2}(B\cap \Pi_{y, \eta})}d\eta\\
& \leq \left\|\mathbb{A}\right\|^s_{L^\infty(y^\perp \cap S^{n-1};\, L^2(B\cap\Pi_{y, \eta}))}\left\| \widetilde{a}_0\right\|^s_{L^s(y^\perp \cap S^{n-1};\, L^2(B\cap\Pi_{y, \eta}))}.
\end{align*}
Let now $\eta_1, \eta_2\in y^\perp\cap S^{n-1}$. From the identity 
\begin{align*}
F(\eta_1)-F(\eta_2)&= \int_{\mathbb{C}}(\chi_{\eta_1}- \chi_{\eta_2})(z) \mathbb{A}(z, \eta_1)\widetilde{a}_0(z, \eta_1)dzd\overline{z}\\
&\quad + \int_{\mathbb{C}} \chi_{\eta_2} (z)(\mathbb{A}(z, \eta_1)- \mathbb{A}(z, \eta_2))\widetilde{a}_0(z, \eta_2) dzd\overline{z}\\
& \quad  +\int_{\mathbb{C}}\chi_{\eta_2} (z)  \mathbb{A}(z, \eta_1)(\widetilde{a}_0(z, \eta_1)-\widetilde{a}_0(z, \eta_2)  ) dzd\overline{z}\\
&\quad  := I_{21}+ I_{22}+ I_{23}
\end{align*}
we deduce
\[
I_2=\int_{y^\perp \cap S^{n-1}}\int_{y^\perp \cap S^{n-1}}\dfrac{|I_{21}+I_{22}+ I_{23}|^s}{|\eta_1-\eta_2|^{n-2+\gamma s}}d\eta_1d\eta_2.
\]
Since the  area of the set $B\cap \Pi_{y, \eta_1}\setminus B\cap \Pi_{y, \eta_2}$ is upper bounded proportionally to  $|\eta_1-\eta_2|$, we get 
\begin{align*}
|I_{21}|&=\left| \int_{B\cap \Pi_{y, \eta_1}\setminus B\cap \Pi_{y, \eta_2}}\mathbb{A}(z, \eta_1) \widetilde{a}_0(z, \eta_1)dzd\overline{z} \right|\\
& \leq \left\|\mathbb{A}(\cdot, \eta_1) \right\|_{L^2(B\cap \Pi_{y, \eta_1}\setminus B\cap \Pi_{y, \eta_2})}\left\| \widetilde{a}_0(\cdot, \eta_1)\right\|_{L^2(B\cap \Pi_{y, \eta_1}\setminus B\cap \Pi_{y, \eta_2})}\\
& \lesssim  \left\| \mathbb{A} \right\|_{L^\infty(\Psi_y(B))}|\eta_1-\eta_2|\left\|\widetilde{a}_0(\cdot, \eta_1) \right\|_{L^{2}(B\cap \Pi_{y, \eta_1})}.
\end{align*}
Since  $n-2 + \gamma s-s <n-2$ for all $\gamma\in(0,1)$, we deduce
\[
\int_{y^\perp \cap S^{n-1}}\int_{y^\perp \cap S^{n-1}}\dfrac{|I_{21}|^s}{|\eta_1-\eta_2|^{n-2+\gamma s}}d\eta_1d\eta_2\lesssim \left\| \widetilde{a}_0\right\|^s_{L^s(y^\perp \cap S^{n-1};\, L^2(B\cap\Pi_{y, \eta}))}.
\]
Now we analize $I_{22}$. Note that  $\mathbb{A}\in W^{1, 2n}_c(\Psi_y(B))$ since $\mathbb{A}\in W^{1, \infty}_c(\Psi_y(B))$. Hence, Morrey's inequality implies
\[
\left\| \mathbb{A} \right\|_{C^{0, \frac{1}{2}}(\Psi_y(B))}\lesssim \left\| \mathbb{A} \right\|_{W^{1, \infty}(\Psi_y(B))}
\]
where the implicit constant depends on $n$ and $B$. In particular, we have
\[
|\mathbb{A}(z, \eta_1)-\mathbb{A}(z, \eta_2)| \lesssim  \left\| \mathbb{A} \right\|_{W^{1, \infty}(\Psi_y(B))} |\eta_1- \eta_2|^{1/2}
\]
for all $\eta_1, \eta_2 \in y^\perp \cap S^{n-1}$ and all $z\in \mathbb{C}$. Combining all these facts, we deduce 
\[
|I_{22}| \lesssim \left\| \mathbb{A} \right\|_{W^{1,\infty}(\Psi_y(B))}|\eta_1-\eta_2|^{1/2}\left\|\widetilde{a}_0(\cdot, \eta_2) \right\|_{L^{2}(B\cap \Pi_{y, \eta_2})}
\]
and since $n-2 + \gamma s-s/2 <n-2$, we get
\[
\int_{y^\perp \cap S^{n-1}}\int_{y^\perp \cap S^{n-1}}\dfrac{|I_{22}|^s}{|\eta_1-\eta_2|^{n-2+\gamma s}}d\eta_1d\eta_2\lesssim \left\| \widetilde{a}_0\right\|^s_{L^s(y^\perp \cap S^{n-1};\, L^2(B\cap\Pi_{y, \eta}))}.
\]

Similar computations give the estimate for $I_{23}$. Morrey's inequality ensures that
\[
\left\| \widetilde{a}_0\right\|_{C^{0, 1- n/s}((\Psi_y(B))} \lesssim \left\| \widetilde{a}_0\right\|_{W^{1, s}((\Psi_y(B))},
\]
where the implicit constant is uniformly bounded depending only on $n$ and $B$. Hence
\[
\int_{y^\perp \cap S^{n-1}}\int_{y^\perp \cap S^{n-1}}\dfrac{|I_{23}|^s}{|\eta_1-\eta_2|^{n-2+\gamma s}}d\eta_1d\eta_2\lesssim\left\| \widetilde{a}_0\right\|_{W^{1, s}\left(y^\perp\cap S^{n-1}; \, L^2(B\cap\Pi_{y, \eta}) \right)}.
\]
We conclude the proof by combining all the previous estimates into \eqref{F_C_gamma}.
\end{proof}

We are now able to interpolate estimates from Lemma \ref{iner_until_infuzz}  and Lemma \ref{lem_interp_L_infty} until $L^\infty(y^\perp \cap S^{n-1})$,  by choosing suitable parameters $p, \gamma$ and $s=2p$. Hence there exists $t_0\in (0,1)$ such that 
\begin{equation}
\begin{aligned}\label{iner_until_infuzx}
&\underset{\eta\in y^\perp \cap S^{n-1}}{\sup}\left|  \int_{B\cap \Pi_{y, \eta}}  (A_t+iA_\theta)(z, \eta)\,e^{i\Phi(z, \eta)}\, \widetilde{a}_0(z, \eta)\, dz\,d\overline{z}\right|         \\
&\quad= \left\|  \int_{B\cap \Pi_{y, \cdot}}  (A_t+iA_\theta)(z, \cdot)\,e^{i\Phi(z, \cdot)}\, \widetilde{a}_0(z, \cdot)\, dz\,d\overline{z} \right\|_{L^{\infty}(y^\perp\cap S^{n-1})}\\
&\quad \lesssim  \left | \log \left \| \Lambda_1^\sharp- \Lambda_2^\sharp \right \| \right |^{-t_0}\left\| \widetilde{a}_0\right\|_{W^{1, 2p}\left(y^\perp\cap S^{n-1}; \, H^1(B\cap\Pi_{y, \eta}) \right)}. 
\end{aligned}
\end{equation}

Before continuing the proof, it will be useful to introduce a change of coordinates to work into the unit ball $B_1(0)$ and its boundary $S^1$ in the complex plane instead of  $B_{R_0}(y_0)\cap \Pi_{y, \eta}$ and $\partial(B_{R_0}(y_0)\cap \Pi_{y, \eta})$. By Lemma \ref{g_e_c_o_s}, Remark \ref{rem_identi_1_i} and \eqref{rad_bound}-\eqref{change_variables_newzxcq}, the following map is well defined
\begin{equation}
\begin{matrix}
\mathcal{T}_{y, \eta}:&\overline{B_{R_0}(y_0)\cap \Pi_{y, \eta}} &\rightarrow&\overline{B_1(0)}\subset \mathbb{C} \\ 
 &z& \mapsto&  \left(e^z-(\left \langle y, y_0 \right \rangle+i \left \langle \eta, y_0 \right) \right)/R_{y, \eta}. 
\end{matrix}
\end{equation}
Moreover, one has
\[
\mathcal{T}_{y, \eta}\left(\partial(B_{R_0}(y_0)\cap \Pi_{y, \eta})\right)= S^1\subset \mathbb{C}.
\]
We now set
\begin{equation}\label{nexcq_ delta_zeta}
\widetilde{\Phi}(\widetilde{z}, \eta)= \Phi(\mathcal{T}_{y, \eta}(\widetilde{z}), \eta),\quad \widetilde{\mathbb{A}}(\widetilde{z}, \eta):= (A_t + i A_\theta)\left( \mathcal{T}^{-1}_{y, \eta}(\widetilde{z}), \eta\right).
\end{equation}
Taking into account \eqref{phi_partia_uno}, a direct application of the chain rule gives us the next equation for each $\eta\in y^{\perp}\cap S^{n-1}$:
\begin{equation}\label{new_anszastew_1}
\partial_{\overline{\widetilde{z}}}\, \widetilde{\Phi}(\cdot, \eta)+ \frac{i}{2} R_{y, \eta}\left(  \overline{\widetilde{z}} \,  {R_{y, \eta}}+ \left \langle y, y_0 \right \rangle-i \left \langle \eta, y_0 \right \rangle \right)^{-1} \widetilde{\mathbb{A}}(\cdot, \eta)=0.
\end{equation}
We consider this equation in the whole complex plane by extending $\widetilde{\mathbb{A}}(\cdot, \eta)$ by zero outside  $B_1(0)$. This extension, still denoted by $\widetilde{\mathbb{A}}(\cdot, \eta)$, belongs to $C^{1+\sigma}_c(\mathbb{C})$. Applying the second estimate from Lemma \ref{delta_estimate} with $\gamma=\sigma$ and $\tilde p=5\sigma^{-1}$, we deduce that $\widetilde{\Phi}(\cdot, \eta)\in C^{2+3\sigma/5}(\mathbb{C})$. Hence its restrictions to $S^1$, denoted by $(\widetilde{\Phi}(\cdot, \eta))|_{S^1}$, also belongs to $C^{2+3\sigma/5}(S^1)$. Furthermore, by Lemma \ref{new_delta_bar_estimate} with $R=1$, one gets
\begin{equation}\label{omh_li}
|| \widetilde{\Phi}||_{W^{1, \infty}(S^{n-2}; H^1(B_1(0))} \lesssim || \widetilde{\mathbb{A}} ||_{W^{1, \infty}(\mathbb{C}\times S^{n-2})} <\infty,
\end{equation}
where the norm on the  left-hand side is defined in Lemma \ref{new_delta_bar_estimate}. Consider $\widetilde{z}:=\mathcal{T}_{y, \eta}(z)$. Stoke's theorem with \eqref{phi_partia_uno} yield
\begin{equation}\label{new_Phi}
\begin{aligned}
& \int_{B\cap \Pi_{y, \eta}}  (A_t+iA_\theta)(z, \eta)\,e^{i\Phi(z, \eta)}\, \widetilde{a}_0(z, \eta)\, dz\,d\overline{z}\\
& = \int_{\partial( B_{R_0}(y_0)\cap \Pi_{y, \eta})}e^{i\Phi(z, \eta)}\, \widetilde{a}_0(z, \eta)\, dz\\
&= \int_{S^1} e^{i\widetilde{\Phi}(\widetilde{z}, \eta)} {R_{y, \eta}} \left( \widetilde{z} \,  {R_{y, \eta}}+ \left \langle y, y_0 \right \rangle+i \left \langle \eta, y_0 \right \rangle\right)^{-1}    \widetilde{a}_0\left( \mathcal{T}^{-1}_{y, \eta}(\widetilde{z}), \eta\right)\, d\widetilde{z}.
\end{aligned}
\end{equation}
This identity will provide us suitable bounds for $e^{i\widetilde{\Phi}}$. Indeed, consider  $\widetilde{a}_0$ as the family $\widetilde{a}_{0, k}$ with $k\in \mathbb{Z}^+$ defined by
\[
\widetilde{a}_{0, k} (z, \eta)= e^z (e^z- (\left \langle y, y_0 \right \rangle+i \left \langle \eta, y_0 \right \rangle))^k  {R_{y, \eta}^{-k-1}}
\]
so that the term accompanied $e^{i\widetilde{\Phi}(\widetilde{z}, \eta)}$ in the last line of \eqref{new_Phi} will be  ${\widetilde{z}}^{\,k}$. In consequence, we have 
\[
 \int_{B\cap \Pi_{y, \eta}}  (A_t+iA_\theta)(z, \eta)\,e^{i\Phi(z, \eta)}\, \widetilde{a}_{0, k}(z, \eta)\, dz\,d\overline{z}= \int_{S^1} e^{i\widetilde{\Phi}(\widetilde{z}, \eta)} \, \widetilde{z}^{\,k}\, d\widetilde{z}.
\]
By \eqref{uniform_bounded_R},  \eqref{iner_until_infuzx} and \eqref{new_Phi}, and since $B_{R_0}(y_0)\cap \Pi_{y, \eta}$ is a bounded set, we get
\[
\left|\int_{S^1} e^{i\widetilde{\Phi}(\widetilde{z}, \eta)} \, \widetilde{z}^{\, k}\, d\widetilde{z}\right|\leq C  \left | \log \left \| \Lambda_1^\sharp- \Lambda_2^\sharp \right \| \right |^{-t_0}k, \quad k\in \mathbb{Z}^+.
\]
Consider the family on $S^1$
\[
\mathbb{F}(\cdot, \eta)=e^{i\widetilde{\Phi}(\cdot, \eta)}, \quad \eta\in y^{\perp}\cap S^{n-1}.
\]
 Because of the discussion just after \eqref{new_anszastew_1}, $\mathbb{F}(\cdot, \eta) \in C^{2+3\sigma/5}(S^1)$ for all $\eta\in y^\perp\cap S^{n-1}$. The previous estimate implies 
\[
|\widehat{\mathbb{F}} (-k, \eta)|\leq C \left | \log \left \| \Lambda_1^\sharp- \Lambda_2^\sharp \right \| \right |^{-t_0} k, \quad k\in \mathbb{Z}^+,
\]
By abuse of notation, $\widehat{\mathbb{F}} (-k, \eta)$ stands for the Fourier coefficient of $\mathbb{F}$ at $-k$ with respect to the first variable and for a fixed $\eta$, see \eqref{fourier_coeff_k}. Combining Remark  \ref{Dini_condition} with Lemma \ref{quantification_complex_argument} applied to $m=2+3\sigma/5$ and $\beta=(\sigma+3)/2$, we deduce the following pointwise representation of $\mathbb{F}$ in $S^1$: 
\[
\mathbb{F}(\cdot, \eta)= F(\cdot, \eta) + \widetilde{\mathbb{F}}(\cdot, \eta),
\]
with
\begin{equation}\label{f_eta_removed_vf}
\left\|\widetilde{\mathbb{F}} (\cdot, \eta) \right\|_{C^{1+\sigma/2}(S^1)}  \lesssim \left\|\widetilde{\mathbb{F}} (\cdot, \eta) \right\|_{H^{(\sigma+3)/2}(S^1)} \lesssim  \left | \log \left \| \Lambda_1^\sharp- \Lambda_2^\sharp \right \| \right |^{-\frac{4\sigma t_0}{3(\sigma+5)}}.
\end{equation}
Here we have used the fact the  $H^{(\sigma+3)/2}(S^1)$ is compactly embedded in $C^{1+\sigma/2}(S^1)$. In addition, the function $F(\cdot, \eta): = (\mathbb{F}  - \widetilde{\mathbb{F}})(\cdot, \eta)$ can be extended holomorphically into $B_1(0)$ (we still denote such an extension by $F(\cdot, \eta)$) and satisfies
\begin{equation}\label{f_eta_removed_vfzi}
\left\|F (, \eta)\right\|_{C^{1+\sigma/2}(S^1)}\lesssim \left\| F (\cdot, \eta)\right\|_{H^{(\sigma+3)/2}(S^1)} \lesssim  \left\| \mathbb{F} (\cdot, \eta)\right\|_{H^{(\sigma+3)/2}(S^1)} <\infty 
\end{equation}
This estimate is uniform in the variable $\eta\in y^{\perp}\cap S^{n-1}$ due to \eqref{omh_li}.\\

 Next, we claim that the extension of $F(\cdot, \eta)$ does not vanish in $B_1(0)$. Thanks to the maximum principle for holomorphic functions, it is enough to prove the non-vanishing property on $S^1$.

%
\begin{lem}\label{lemma_bun_acot}
There exist universal constants $C_1, C_2>0$ such that 
\[
|F(\widetilde{z}, \eta)| \geq C_1, \; |\widetilde{z}|=1,\, \eta\in y^\perp\cap S^{n-1},
\]
whenever $|| \Lambda_1^\sharp- \Lambda_2^\sharp ||\leq C_2$.
\end{lem}
\begin{proof}
We set $C_0:=|| \widetilde{\Phi} ||_{L^{\infty}(\mathbb{C}: L^{\infty}(y^\perp \cap S^{n-1}))}$, which is finite by using Lemma \ref{delta_estimate} in \eqref{new_anszastew_1}. Consider $\eta\in y^\perp\cap S^{n-1}$ and $|\widetilde{z}|=1$. By \eqref{f_eta_removed_vf}, we get
\[
|\widetilde{\mathbb{F}}(\widetilde{z}, \eta)|\leq C \left | \log \left \| \Lambda_1^\sharp- \Lambda_2^\sharp \right \| \right |^{-\frac{4\sigma t_0}{3(\sigma+5)}}\leq \frac{1}{2} e^{-C_0}\leq \frac{1}{2} |\mathbb{F}(\widetilde{z}, \eta)|,
\]
where the intermediate inequality holds if, for instance
\[
|| \Lambda_1^\sharp- \Lambda_2^\sharp ||\leq e^{-\left(2Ce^{C_0} \right)^{\frac{3(\sigma+5)}{4\sigma t_0}}}:=C_2.
\]
We end the proof by noting that
\begin{align*}
|F(\widetilde{z}, \eta)|= | (\mathbb{F}  - \widetilde{\mathbb{F}})(\widetilde{z}, \eta)|\geq |\mathbb{F}(\widetilde{z}, \eta)| - |\widetilde{\mathbb{F}}(\widetilde{z}, \eta)|\\
\geq \frac{1}{2}|\mathbb{F}(\widetilde{z}, \eta) | 
 \geq \frac{1}{2} e^{-C_0}:=C_1.
\end{align*}
\end{proof}
As an immediate consequence, the logarithm of the extension of $F(\cdot, \eta)$ is well defined in $B_1(0)$. Let $H(\cdot, \eta)$ be the holomorphic function in $B_1(0)$ and continuous in $\overline{B_1(0)}$ so that $F(\cdot, \eta)= e^{H(\cdot, \eta)}$. Combining the identity
\[
e^{i \widetilde{\Phi}(\widetilde{z}, \eta)-H(\widetilde{z}, \eta)}= 1+ \widetilde{\mathbb{F}}(\widetilde{z}, \eta)/(e^{i \widetilde{\Phi}(\widetilde{z}, \eta)}- \widetilde{\mathbb{F}}(\widetilde{z}, \eta)), \; |\widetilde{z}|=1, \, \eta\in y^\perp\cap S^{n-1}
\]
with \eqref{f_eta_removed_vf}-\eqref{f_eta_removed_vfzi} and using Lemma \ref{lemma_bun_acot}, we deduce
\begin{equation}\label{Z1x1}
\begin{aligned}
&| \widetilde{\mathbb{F}}(\widetilde{z}, \eta)/(e^{i \widetilde{\Phi}(\widetilde{z}, \eta)}- \widetilde{\mathbb{F}}(\widetilde{z}, \eta))|&\\
&\lesssim  \left | \log \left \| \Lambda_1^\sharp- \Lambda_2^\sharp \right \| \right |^{-\frac{4\sigma t_0}{3(\sigma+5)}}, \; |\widetilde{z}|=1, \, \eta\in y^\perp\cap S^{n-1}.&
\end{aligned}
\end{equation}
Hence, Lemma \ref{quantification_complex_argument_5} implies 
\[
\begin{aligned}
&|i \widetilde{\Phi}(\widetilde{z}, \eta)-H(\widetilde{z}, \eta)|&\\
&\lesssim  2 \left | \log \left \| \Lambda_1^\sharp- \Lambda_2^\sharp \right \| \right |^{-\frac{4\sigma t_0}{3(\sigma+5)}},\; |\widetilde{z}|=1, \, \eta\in y^\perp\cap S^{n-1}.&
\end{aligned}
\]
We now choose $\widetilde{a_0}$ (with $\widetilde{z}\in B_1(0)$ and $\eta\in y^\perp\cap S^{n-1}$) in \eqref{new_Phi} as
\begin{align*}
\widetilde{a_0}(\mathcal{T}_{y, \eta}^{-1}(\widetilde{z}), \eta)= H(\widetilde{z}, \eta)\, e^{-H{(\widetilde{z}, \eta)}} \widetilde{b_0}(\mathcal{T}_{y, \eta}^{-1}(\widetilde{z}), \eta), \quad \partial_{\overline{\widetilde{z}}} \, \widetilde{b_0}(\mathcal{T}_{y, \eta}^{-1}(\widetilde{z}), \eta)=0.
\end{align*}
Combining \eqref{nexcq_ delta_zeta}-\eqref{new_anszastew_1} and Stoke's theorem, we get 
\begin{equation}\label{I_II_III}
\begin{aligned}
& \int_{B\cap \Pi_{y, \eta}}  (A_t+iA_\theta)(z, \eta) \widetilde{b}_0(z, \eta)\, dz\,d\overline{z}\\
&= \frac{1}{2} \int_{|\widetilde{z}|\leq 1} R_{y, \eta}^2  | \widetilde{z} \,  {R_{y, \eta}}+ \left \langle y, y_0 \right \rangle+i \left \langle \eta, y_0 \right \rangle|^{-2} \widetilde{\mathbb{A}}(\widetilde{z}, \eta) \widetilde{b}_0(z, \eta)\, d\widetilde{z}\,d\overline{\widetilde{z}}\\
&=i \int_{|\widetilde{z}|=1} {R_{y, \eta}} \left( \widetilde{z} \,  {R_{y, \eta}}+ \left \langle y, y_0 \right \rangle+i \left \langle \eta, y_0 \right \rangle\right)^{-1} \widetilde{\Phi}(\widetilde{z}, \eta) \widetilde{b}_0(z, \eta)d\widetilde{z}\\
&:= I + II + III,
\end{aligned}
\end{equation}
where
\begin{align*}
I&=  \int_{|\widetilde{z}|=1} e^{i\widetilde{\Phi}(\widetilde{z}, \eta)} {R_{y, \eta}} \left( \widetilde{z} \,  {R_{y, \eta}}+ \left \langle y, y_0 \right \rangle+i \left \langle \eta, y_0 \right \rangle\right)^{-1}\\
&\qquad \qquad\qquad \qquad \qquad \qquad  \times  H(\widetilde{z}, \eta) e^{-H(\widetilde{z}, \eta) }  \widetilde{b}_0\left( \mathcal{T}^{-1}_{y, \eta}(\widetilde{z}), \eta\right)\, d\widetilde{z}  \\
II&= -  \int_{|\widetilde{z}|=1}  {R_{y, \eta}} \left( \widetilde{z} \,  {R_{y, \eta}}+ \left \langle y, y_0 \right \rangle+i \left \langle \eta, y_0 \right \rangle\right)^{-1} \\
&\qquad \qquad\qquad \qquad \qquad \qquad  \times  (H(\widetilde{z}, \eta) -i\widetilde{\Phi}(\widetilde{z}, \eta) ) \widetilde{b}_0\left( \mathcal{T}^{-1}_{y, \eta}(\widetilde{z}), \eta\right)\, d\widetilde{z} \\
III&=  -  \int_{|\widetilde{z}|=1}  \dfrac{ \widetilde{\mathbb{F}}(\widetilde{z}, \eta)}{e^{i \widetilde{\Phi}(\widetilde{z}, \eta)}- \widetilde{\mathbb{F}}(\widetilde{z}, \eta)} {R_{y, \eta}} \left( \widetilde{z} \,  {R_{y, \eta}}+ \left \langle y, y_0 \right \rangle+i \left \langle \eta, y_0 \right \rangle\right)^{-1} \\
&\qquad \qquad\qquad \qquad \qquad \qquad  \times  H(\widetilde{z}, \eta) \widetilde{b}_0\left( \mathcal{T}^{-1}_{y, \eta}(\widetilde{z}), \eta\right)\, d\widetilde{z}.
\end{align*}

We choose  $\widetilde{b_0}(\mathcal{T}_{y, \eta}^{-1}(\widetilde{z}), \eta)=\left( \widetilde{z} \,  {R_{y, \eta}}+ \left \langle y, y_0 \right \rangle+i \left \langle \eta, y_0 \right \rangle\right)^{i\lambda}$ with $\lambda>0$. We remark that this family of functions is analytic in the complex set $B\cap \Pi_{y, \eta}$ for all $(y, \eta)\in \partial_+ S(S^{n-1}_{>\beta^\prime})$, because $\lambda$ is real and the complex logarithm is well defined in the complex set $B\cap \Pi_{y, \eta}$ for all $(y, \eta)\in \partial_+ S(S^{n-1}_{>\beta^\prime})$. Taking into account \eqref{Z1x1}-\eqref{I_II_III} and \eqref{iner_until_infuzx} with \eqref{omh_li}, we obtain the following estimate for every $y\in \partial S^{n-1}_{>\beta^\prime}$:
\begin{align*}
|I| & \lesssim  \left | \log \left \| \Lambda_1^\sharp- \Lambda_2^\sharp \right \| \right |^{-t_0}\left\| H e^{-H} \right\|_{W^{1, 2p}\left(y^\perp\cap S^{n-1}; \, H^1(B_1(0)) \right)}\\
& =  \left | \log \left \| \Lambda_1^\sharp- \Lambda_2^\sharp \right \| \right |^{-t_0}\left\| { \widetilde{\Phi}}^{-1} \log \widetilde{\Phi}  \right\|_{W^{1, 2p}\left(y^\perp\cap S^{n-1}; \, H^1(B_1(0)) \right)}\\
& \lesssim  \left | \log \left \| \Lambda_1^\sharp- \Lambda_2^\sharp \right \| \right |^{-t_0}\left\| \widetilde{\Phi}  \right\|_{W^{1, \infty}\left(y^\perp\cap S^{n-1}; \, H^1(B_1(0)) \right)}  \lesssim  \left | \log \left \| \Lambda_1^\sharp- \Lambda_2^\sharp \right \| \right |^{-t_0},\\
|II| &\lesssim  \left\| H-i\widetilde{\Phi}\right\|_{L^{\infty}(S^1\times y^\perp\cap S^{n-1})} \lesssim  \left | \log \left \| \Lambda_1^\sharp- \Lambda_2^\sharp \right \| \right |^{-\frac{4\sigma t_0}{3(\sigma+5)}},\\
|III| &\lesssim  \left\|\widetilde{\mathbb{F}}/(e^{i \widetilde{\Phi}}- \widetilde{\mathbb{F}}) \right\|_{L^{\infty}(S^1\times y^\perp\cap S^{n-1})} \lesssim  \left | \log \left \| \Lambda_1^\sharp- \Lambda_2^\sharp \right \| \right |^{-\frac{4\sigma t_0}{3(\sigma+5)}}.
\end{align*}
We end the proof of Theorem \ref{remov_esti_phi_1} by combining Lemma \ref{new_delta_bar_estimate} and \eqref{new_anszastew_1}.
\end{appendix}

\section*{Acknowledgements} This work was completed during Leo Tzou's visit to Jyv\"aksyl\"a and the author wishes to thank Mikko Salo and University of Jyv\"aksyl\"a for their hospitality. Leo Tzou is partially supported by ARC DP190103302 and ARC DP190103451. L.P-M., \ was supported by the Academy of Finland (Centre of Excellence in Inverse Modelling and Imaging, grant numbers 312121 and 309963) and by the European Research Council under Horizon 2020 (ERC CoG 770924), and partially supported by Spanish government predoctoral grant BES-2012-058774 and project MTM2011-28198. A.R.,\ was supported by Spanish Grant MTM2017-85934-C3-2-P.

\end{document}